\documentclass{article}

\usepackage[titletoc, title]{appendix}
\usepackage[english]{babel}
\usepackage[mathscr]{eucal}
\usepackage[margin = 3.5 cm]{geometry}
\usepackage[utf8]{inputenc}
\usepackage[dvipsnames]{xcolor}
\usepackage{amsfonts, amsmath, amssymb, amsthm, bbm}
\usepackage{enumitem}
\usepackage{longtable}
\usepackage{mathrsfs}
\usepackage{caption}
\usepackage{hyperref}
\hypersetup{
    colorlinks = true,
    linkcolor = black,
    filecolor = black,
    urlcolor = black,
    citecolor = black
}
\usepackage{tikz}
\usetikzlibrary{automata, positioning, arrows.meta, quotes}
\tikzset{every node/.style = {font = \tiny, sloped}}
\tikzset{every state/.style={rectangle, minimum size = 0pt, draw = none, font = \normalsize}}

\theoremstyle{plain}
\newtheorem{theorem}{Theorem}[section]

\newtheorem*{theorem*}{Theorem}
\newtheorem*{corollary*}{Corollary}
\newtheorem*{lemma*}{Lemma}
\newtheorem*{proposition*}{Proposition}

\theoremstyle{definition}

\newtheorem{example}[theorem]{Example}
\newtheorem*{definition*}{Definition}
\newtheorem*{example*}{Example}

\theoremstyle{remark}
\newtheorem{remark}[theorem]{Remark}

\newtheorem*{remark*}{Remark}
\newtheorem*{assumption*}{Assumption}

\newcommand{\preord}[1]{\preccurlyeq_{#1}}
\newcommand{\npreord}[1]{\not\preccurlyeq_{#1}}
\newcommand{\rn}{\mathcal{G}}
\newcommand{\spc}{\mathscr{S}}
\newcommand{\cmp}{\mathscr{C}}
\newcommand{\rct}{\mathscr{R}}
\newcommand{\nat}{\Z_{\geq 0}}
\renewcommand{\ss}{\mathcal{S}}

\newcommand{\st}{ \ : \ }
\newcommand{\nrm}[1]{\left\|#1\right\|}
\newcommand{\abs}[1]{\left|#1\right|}
\newcommand{\rb}[1]{\left(#1\right)}
\newcommand{\brb}[1]{\big(#1\big)}
\newcommand{\Brb}[1]{\Big(#1\Big)}

\newcommand{\bsb}[1]{\big[#1\big]}

\newcommand{\bbsb}[1]{\bigg[#1\bigg]}
\newcommand{\cb}[1]{\left\{#1\right\}}
\newcommand{\bcb}[1]{\big\{#1\big\}}

\newcommand{\Z}{\mathbb{Z}}

\newcommand{\R}{\mathbb{R}}

\renewcommand{\P}{\mathbb{P}}

\newcounter{case}
\renewcommand{\thecase}{\Roman{case}}
\newcounter{hypa}
\renewcommand{\thehypa}{\textup{(}\textit{a\arabic{hypa}}\textup{)}}
\newcounter{hypb}
\renewcommand{\thehypb}{\textup{(}\textit{b\arabic{hypb}}\textup{)}}

\title{\textbf{Computable sufficient conditions for comparing stochastic reaction networks with mass-action \\ or monotone kinetics}}
\author{Daniele Cappelletti, Giulio Cuniberti, and Paola Siri \\
    \small Department of Mathematical Sciences, Politecnico di Torino, Turin, Italy \\
    \small \texttt{\string{daniele.cappelletti,giulio.cuniberti,paola.siri\string}@polito.it}}
\date{September 2026}

\begin{document}

\maketitle

\begin{abstract}
    Stochastic reaction networks are mathematical models with a wide range of applications in biochemistry, ecology, and epidemiology, and are often complex to analyze. Except for some special cases, it is generally difficult to predict how the abundances of all considered species evolve over time. A possible approach to address this issue is to develop tools to compare the model under study with a similar one whose behavior is better understood. The main contribution of our work is to provide direct and computable conditions that can be used to ensure the existence of an ordered coupling between two stochastic reaction networks with monotone kinetics and to identify which parameter changes in a given model lead to an increase or decrease in the count of certain species. We also make available an algorithm that implements our theory, and we illustrate it with several applications.
\end{abstract}


\section{Introduction}

Markov chain models for reaction networks have been studied systematically since the 1970s (see \cite{kurtz1972relationship} and \cite{gillespie1976general}, for example). Such stochastic processes count the number of individuals of different interacting \emph{species} that are present at a certain time $t$. The transitions are given by a finite set of \emph{reactions} that describe recombination, production, and degradation of the species. Stochastic reaction networks have many applications in biochemistry (see, for example, \cite{jeon2018mathematical}), synthetic biology (e.g.\ \cite{briat2016antithetic}), ecology (e.g.\ \cite{veloz2021reaction}), chemistry (e.g.\ \cite{sabbioni2024final}), and epidemiology (e.g.\ \cite{khudabukhsh2020survival}). In these contexts, both dynamic and long-term behaviors are of interest and are often difficult to analyze directly, hence having tools to compare a given model with a simpler one can be useful. Also, it is important for experimentalists to understand which changes in the model parameters lead to an average increase or decrease in some of the species abundances. A suitable mathematical framework for this context is the theory of stochastic orderings. To introduce it, let us consider the simple reversible conversion
\begin{equation}
\label{eq:int}
    \text{S} \rightleftharpoons \text{P},
\end{equation}
and imagine increasing the rate of the forward reaction while decreasing the rate of the backward one. Intuitively, one expects the resulting system to express more P and less S at all times. To make it formal, if $X = \rb{X_\text{S}, X_\text{P}}$ denotes the continuous-time Markov chain associated with the original system and $Y = \rb{Y_\text{S}, Y_\text{P}}$ the process modeling the modified system, one would like to guarantee that $X$ and $Y$ can be constructed in such a way that $X_\text{S}(t) \geq Y_\text{S}(t)$ and $X_\text{P}(t) \leq Y_\text{P}(t)$ for all $t \geq 0$, with probability one. This is equivalent to saying that the laws of the original and modified processes are in the \emph{usual stochastic order} (up to changing the sign of the S coordinate), as referred to in the literature, and makes it formally precise what we mean by increase or decrease of a species abundance. The same idea and construction for $\text{S} \rightleftharpoons \text{P}$ apply to larger and more sophisticated reaction networks, with the aim of studying how certain changes in the reaction rates may influence the number of molecules of each species.

As can be seen from classic stochastic ordering theory textbooks such as \cite{muller2002comparison} and \cite{shaked2007stochastic}, there are only a few ready-to-use results for the markovian framework, mostly dating back several decades and not directly and easily applicable to stochastic reaction networks; see, for example, \cite{kamae1977stochastic} and \cite{massey1987stochastic}. To the best of our knowledge, the first application of ideas and tools from the field of stochastic orderings to the specific setting of reaction network models occurred very recently, in \cite{campos2023comparison}. There, the results from \cite{kamae1977stochastic} and \cite{massey1987stochastic} were summarized, and their unsuitability for stochastic reaction networks was discussed. A new theory was developed, and several real-world scenarios where it can be fruitfully applied were given. As pointed out in \cite{campos2023comparison}, before the publication of their paper there were already some works on sensitivity analysis of mass-action models (see \cite{gunawan2005sensitivity} and \cite{gupta2014sensitivity}), pursuing similar general goals but more specialized in scope, since they only accounted for local changes of the parameters. Moreover, the results of \cite{campos2023comparison} have already been used in more recent papers, for example in \cite{bruno2024analysis} to investigate the effect of chromatin modification dynamics on epigenetic cell memory and in \cite{hirsch2023error} to study how much Hill-function-based approximations deviate from standard gene-regulation mass-action models. The setting in \cite{campos2023comparison} is actually slightly more general than stochastic reaction networks. While powerful, their results remain however quite abstract:  the hypotheses involve inequalities between transition rates that must hold for all admissible pairs of states. Although these can sometimes be verified by finitely many analytical arguments, such a verification can be difficult in concrete applications.
  
The goal of this work is to strengthen the applicability of the framework presented in \cite{campos2023comparison} on concrete reaction networks, allowing a broader group of researchers in the field to use comparison tools to better understand the behavior of the models they are studying. Our contribution is twofold. First, we provide a new formulation of \cite[Theorem~3.1]{campos2023comparison} that generalizes matrix preorders to a broader class of binary relations, and that accommodates potentially explosive chains; our proof is different and considerably shorter. Second -- and most importantly -- we give a finite set of explicit linear conditions on the model parameters: unlike the conditions in \cite{campos2023comparison}, these do not need to be checked separately for individual pairs of states. They can be checked by a computer to guarantee the validity of the hypotheses of the general theorem in the specific setting of stochastic reaction networks with monotone kinetics, which include mass-action kinetics as a particular case. We explicitly write an algorithm implementing this check which is able, for instance, to prove the existence of a stochastic ordering for \eqref{eq:int} -- see Example~\ref{ex:rev}. As we explain in Section~\ref{sub:imp}, our algorithm is amenable to parallelization and relies on existing methods developed in the context of linear programming. We then apply our code to detect stochastic orderings in several examples, which are more complex than \eqref{eq:int} and would be harder to analyze by hand. Finally, to further motivate the study of this technique, we illustrate in Example~\ref{ex:erg} how our methods can be useful to detect ergodicity.

\section{Preliminaries}

Even though our first and more general result -- Theorem~\ref{thm:pcc} -- is stated for arbitrary Continuous-Time Markov Chains (CTMCs in the following, see \cite{norris1998markov} as a reference), our main interest is to particularize it to stochastic models for reaction networks -- as we will do in Theorem~\ref{thm:lsc}. Having this in mind, we will start by recalling what reaction networks are and describing the stochastic framework in which they are usually modeled and studied: for a comprehensive treatment of the subject we refer the reader to \cite{anderson2015stochastic}. Before focusing on reaction networks, however, let us briefly introduce the general notation that will be used throughout this paper.

\subsection{Notation and terminology}
\label{sub:not}

In the following, we will use the notations $\nat = \cb{x \in \Z \st x \geq 0}$ and $\R_{\geq 0} = \cb{x \in \R \st x \geq 0}$ for the sets of non-negative integers and non-negative real numbers. For $x, y \in \R$, let $x \wedge y$ denote the minimum of $x$ and $y$. For all $x, n \in \nat$ define the \textbf{falling factorial} $\rb{x}_n$ as
\begin{equation*}
    \rb{x}_n = \underbrace{x(x-1) \dotsm (x-n+1)}_{n \ \text{factors}} = \prod_{k = 1}^n (x-k+1),
\end{equation*}
where we consider the empty product to be $1$, leading to $\rb{x}_0 = 1$. Note that when $x<n$, one of the factors is $0$, hence $\rb{x}_n = 0$. 

A \textbf{binary relation} on a set $E$ is a subset of $E \times E$. A \textbf{preorder} is a binary relation that is reflexive and transitive, but not necessarily antisymmetric. Every matrix $M \in \R^{m,d}$ with $m \geq 1$ defines the preorder $\preord{M}$ on $\R^d$
\begin{equation}
\label{eq:preord}
    x \preord{M} y \ \ \text{if and only if} \ \ M(x-y) \leq 0
\end{equation}
for each $x,y \in \R^d$, where the above inequality in $\R^m$ is intended component-wise, and $x \preord{M} y$ is a shorthand notation for $(x,y) \! \in \, \preord{M}$. This relation can be restricted to a subset $\Gamma \subseteq \R^d$, and the matrix giving rise to a (restricted) matrix preorder is not unique. For instance, $\preord{M}$ is invariant with respect to any multiplication of one of the rows of $M$ by a positive scalar. This observation motivates the definition of the $\Gamma$-dependent equivalence relation on the set of $\R$-valued matrices with $d$ columns
\begin{equation}
\label{eq:equiv}
    M \sim_{\Gamma} M' \ \ \text{if and only if} \ \ \forall x,y \in \Gamma \ \ x \preord{M} y \Leftrightarrow x \preord{M'} y
\end{equation}
for any $M \in \R^{m,d}$ and $M' \in \R^{m',d}$, with $m,m' \geq 1$.

For any matrix $M$, we will denote by $M^\intercal$ its transpose and by $M_{i,*}$ its $i$-th row. For all $x = \rb{x_1,\dots,x_d}, \ y = \rb{y_1,\dots,y_d} \in \R^d$ denote $\textup{supp}(x) = \bcb{j \in \cb{1,\dots,d} \st x_j \neq 0}$ the support of $x$, $x^\pm = \rb{x_1^\pm,\dots,x_d^\pm}$ the positive/negative part of $x$, $x \odot y = (x_1y_1,\dots,x_dy_d)$ the component-wise product of $x$ and $y$, and $\langle x,y \rangle$ their scalar product. For $x \in \Z^d$, denote by $\text{GCD}(x)$ the greatest common divisor of $x_1,\dots,x_d$, using the convention $\text{GCD}(0,\dots,0)=0$. We will also use the notation $\text{supp}(\cdot)$ for the support of a generic function, i.e.\ the subset of the domain whose elements are not mapped to zero. Lastly, the canonical basis of $\R^d$ will be written $\cb{e_1,\dots,e_d}$.

Given a càdlàg stochastic process $X = \brb{X(t)}_{t \geq 0}$ on a probability space $(\Omega,\mathcal F,\P)$, taking values in $E \subseteq \Z^d$, denote by $\tau^X$ its \textbf{explosion time}, i.e.\ the random variable
\begin{equation}
\label{eq:exp}
    \tau^X = \inf \bcb{t \geq 0 \st \forall \Gamma \subseteq E \ \text{finite} \ \exists s < t \ \text{with} \ X(s) \notin \Gamma}.
\end{equation}
When $\P\rb{\tau^X = \infty} = 1$, the process $X$ is said to be \textbf{non-explosive}.

\subsection{Stochastic reaction networks}
\label{sub:srn}

A reaction network is a mathematical object historically conceived to describe how different chemical species, located in a given system, can interact and transform into each other (and therefore often referred to in the literature as a \emph{chemical} reaction network). In this original setting, we can think of it as a list of available chemical reactions, that additionally encodes their possible sharing of reactants or products. More generally, a \emph{species} appearing in a reaction network can represent any homogeneous group of entities that can gather in a certain quantity, meet with individuals belonging to other species, and together be (completely or partially) converted -- thanks to some natural phenomenon generically called \emph{reaction} -- into a different combination of units of other species. Formally, a \textbf{reaction network} is defined as a triple $\rn = \rb{\spc, \cmp, \rct}$, where:
\begin{itemize}
    \item $\spc$ is the set of \textbf{species}, i.e.\ a finite set of symbols. We will denote its cardinality by $d = \abs{\spc}$ and refer to it as the \textbf{dimension} of the network;
    \item $\cmp$ is the set of \textbf{complexes},  i.e.\ a finite set of formal linear combinations of species with non-negative integer coefficients; 
    \item $\rct$ is the set of \textbf{reactions}, i.e.\ a finite set of ordered pairs of complexes. Normally, a generic reaction $r \in \rct$ is represented as $\nu_r \to \nu_r'$, where $\nu_r$ is called \textbf{source complex} and $\nu_r'$ \textbf{product complex} of $r$.
\end{itemize}

As is typical, we assume that the network is non-redundant: each species appears in at least one complex, each complex appears in at least one reaction, and each reaction is between distinct complexes. In this case, the reaction network is uniquely identified by the \textbf{reaction graph}, i.e.\ the graph whose nodes are $\cmp$ and whose edges are $\rct$. Moreover, each complex can be identified with an element of $\nat^d \simeq \nat^\spc$ as follows: given an ordering of the species, the $i$-th component of the corresponding vector is the coefficient of the $i$-th species. This allows the computation of the difference $\xi_r = \nu'_r - \nu_r \in \Z^d$, called the \textbf{reaction vector} of $r$, which describes the net change in the abundances of all species resulting from reaction $r$.

\vskip 10 pt

\begin{example}
\label{ex:mm1}
    Let us illustrate the above definition of a reaction network with the \textbf{Michaelis--Menten} model, first proposed by \cite{michaelis1913kinetik} in the year 1913 (see also \cite{kang2019quasi}). It is one of the networks that will be used to test the algorithm presented in this paper, in Example~\ref{ex:mm2}. It describes the conversion of a substrate S to a product P, catalyzed by an enzyme E reversibly binding to the substrate in an intermediate enzyme-substrate complex C and subsequently unbinding to produce P. The reaction network is characterized by the sets
    \begin{align*}
        \spc & = \cb{\text{S}, \ \text{E}, \ \text{C}, \ \text{P}}, \\
        \cmp & = \cb{\text{S}+\text{E}, \ \text{C}, \ \text{E}+\text{P}}, \\
        \rct & = \cb{\text{S}+\text{E} \to \text{C}, \ \text{C} \to \text{S}+\text{E}, \ \text{C} \to \text{E}+\text{P}},
    \end{align*}
    its dimension is $4$, and it can be depicted by the graph
    \begin{center}
    \begin{tikzpicture}
        \node[state] (S+E)  at (0,0)    {$\text{S}+\text{E}$}; 
        \node[state] (C)    at (2,0)    {$\text{C}$}; 
        \node[state] (E+P)  at (4,0)    {$\text{E}+\text{P}.$};
        \path[-{Stealth[harpoon]}] 
        ([yshift = 2 px]S+E.east)  edge    node{} ([yshift = 2 px]C.west)
        ([yshift = -2 px]C.west)   edge    node{} ([yshift = -2 px]S+E.east);
        \path[-{Stealth}]
        (C)       edge                    node{} (E+P);
    \end{tikzpicture}
    \end{center}
    Bearing in mind the embedding of $\cmp$ in $\nat^4$ introduced above with the species ordered as S, E, C, P, the reaction vectors for this model are computed as
    \begin{align*}
        \xi_{\text{S}+\text{E} \to \text{C}} & = \nu_{\text{S}+\text{E} \to \text{C}}' - \nu_{\text{S}+\text{E} \to \text{C}} = (0,0,1,0) - (1,1,0,0) = (-1,-1,1,0), \\
        \xi_{\text{C} \to \text{S}+\text{E}} & = \nu_{\text{C} \to \text{S}+\text{E}}' - \nu_{\text{C} \to \text{S}+\text{E}} = (1,1,0,0) - (0,0,1,0) = (1,1,-1,0), \\
        \xi_{\text{C} \to \text{E}+\text{P}} & = \nu_{\text{C} \to \text{E}+\text{P}}' - \nu_{\text{C} \to \text{E}+\text{P}} = (0,1,0,1) - (0,0,1,0) = (0,1,-1,1).
    \end{align*}
\end{example}

A reaction network is only of limited interest on its own. What is really useful about it is the fact that one can easily build a dynamical model on top of it. To do so, a \textbf{stochastic kinetics} is defined as a collection of \textbf{rate functions} (or \textbf{propensities}) $\lambda_r : E \subseteq \nat^d \longrightarrow \R_{\geq 0}$, one for each reaction $r \in \rct$, such that $\lambda_r(x) = 0$ whenever $x+\xi_r \notin E$. A \textbf{Stochastic Reaction Network} (\textbf{SRN}) is then given by two objects:
\begin{enumerate}
    \item a reaction network $\rn = (\spc,\cmp,\rct)$;
    \item a CTMC $X = \brb{X(t)}_{t \geq 0}$ with state space $E \subseteq \nat^d \simeq \nat^\spc$ and transition rates 
    \begin{equation}
    \label{eq:rma}
        q(x, x') = \sum\limits_{\substack{r \in \rct \\ \xi_r = x'-x}} \lambda_r(x), \ \ \ x \neq x'.
    \end{equation}
\end{enumerate}
Note that, $\rct$ being finite, SRNs cannot have instantaneous states: for any $x \in E$, its total exit rate $-q(x,x) = \sum\limits_{x' \neq x} q(x,x') = \sum\limits_{r \in \rct} \lambda_r(x)$ is necessarily finite.

A common choice of rate functions is given by the \textbf{Mass-Action Kinetics} (\textbf{MAK}): the rate of a reaction is directly proportional to the number of possible combinations of the molecules that give rise to it. Formally,
\begin{equation}
\label{eq:mak}
    \lambda_r(x) = \kappa_r \prod_{j = 1}^d \rb{x_j}_{\nu_{r,j}} = \kappa_r \prod_{j \in \text{supp}(\nu_r)} (x_j)_{\nu_{r,j}},
\end{equation}
where $K = \rb{\kappa_r}_{r \in \rct}$ is a collection of non-negative constants, called \textbf{rate constants}, which are typically written on top of the arrows in the reaction graph.

More generally, we will consider \textbf{monotone kinetics} of the form
\begin{equation}
\label{eq:mok}
    \lambda_r(x) = f_r(\kappa_r, x),
\end{equation}
where $f_r : \R_{\geq 0} \times \nat^d \longrightarrow \R_{\geq 0}$ depends on $x$ only through the reactant species, i.e.\ through the components indexed by $\text{supp}(\nu_r)$, and satisfies, for all $\kappa_r^X, \kappa_r^Y \geq 0$ and $x, y \in \nat^d$,
\begin{align*}
    \kappa_r^X \leq \kappa_r^Y, \ x_j \leq y_j \ \forall j \in \text{supp}(\nu_r) \implies f_r(\kappa_r^X, x) \leq f_r(\kappa_r^Y, y).
\end{align*}
We additionally assume $f_r(\kappa_r,x) = 0$ whenever $x+\xi_r \notin \nat^d$, as for a generic rate function $\lambda_r$. The conditions above are inspired by \cite[Definition~2.1]{vassena2023symbolic}, but we also require monotonicity with respect to the parameter $\kappa_r$. We still refer to $\kappa_r \geq 0$ as the \emph{rate constant} associated with reaction $r$ and define the vector $K = \rb{\kappa_r}_{r \in \rct}$: once the functions $f_r$ are specified, the collection $K$ gives a full parametric description of the model. In Theorem~\ref{thm:lsc} this will become important, as the functions $f_r$ will be kept fixed across the two compared models and only the rate constants will vary. Mass-action kinetics is a particular case of monotone kinetics, with $f_r$ given by \eqref{eq:mak}, which is clearly monotone with respect to both $\kappa_r$ and all reactant species counts. Other classical examples include Michaelis--Menten and activating Hill rates (see \cite[Chapter~3]{ingalls2013mathematical} for their expressions): here $\kappa_r$ is one of the kinetic parameters, chosen so that the rate function is non-decreasing with respect to it, while all other parameters are kept fixed.

\vskip 10 pt

\begin{remark}
    In the literature, it is usually assumed that the rate constants are positive. The results of this paper, however, hold in the more general setting of non-negative rate constants. This will become relevant as discussed in Remark~\ref{rem:dif}.
\end{remark}

We conclude this preliminary section with the notion of \textbf{stoichiometric subspace} of a reaction network, defined as
\begin{equation*}
    \ss = \text{span}\cb{\xi_r}_{r \in \rct}.
\end{equation*}
Its role is fundamental because, once the dynamics is initiated at $x^0$, it will be forever trapped in the set $\rb{x^0 + \ss} \cap \nat^d$, called the \textbf{stoichiometric compatibility class} of $x^0$. Usually, the dimension of $\ss$ is denoted by the letter $s$, so that the orthogonal complement $\ss^\perp$ has dimension $d-s$. Each element of $\ss^\perp$ encodes a \textbf{conservation law} for the system, that is, a particular linear combination of the abundances of all species that is not changed by the occurrence of any reaction of the network. For each reaction network considered in this paper, we will always assume that a basis for $\ss^\perp$ has been chosen, and we will list its elements by rows in a matrix $C \in \R^{d-s,d}$. In other words, $C$ can be any matrix with $d-s$ rows and $d$ columns such that $\ker(C) = \ss$.

\section{Pathwise comparison of CTMCs}

The following theorem applies to all CTMCs with values in $\Z^d$. Examples of such processes include MAK SRNs, but at the moment we do not need to assume the presence of any underlying reaction network structure. This result is essentially a stronger version of  \cite[Theorem~3.1]{campos2023comparison}: we allow the CTMCs to potentially explode in finite time and replace the matrix preorder $\preord{M}$ on their state space with a more general binary relation $\rho$ (we will switch back to $\preord{M}$ for the majority of practical applications shortly, but only after showing, in Example~\ref{ex:pcc}, that working with a weaker relation can be useful). Despite the greater generality, our proof is shorter than the one given in \cite{campos2023comparison}. Both are based on coupling arguments, but the coupling constructions are different.

\vskip 10 pt

\begin{samepage}
\begin{theorem}
\label{thm:pcc}
Consider a set $E \subseteq \mathbb{Z}^d$ and two rate matrices $q^X, q^Y : E \times E \longrightarrow \mathbb{R}$ (this implies finite diagonal elements), extended to functions $\bar{q}^X, \bar{q}^Y : E \times \mathbb{Z}^d \longrightarrow \mathbb{R}$ by setting them equal to zero on $E \times (\mathbb{Z}^d \setminus E)$. Let $\rho \subseteq E \times E$ be a binary relation on $E$ such that for each $(x,y) \in \rho$ and $\xi \in \Z^d$ we have
\begin{enumerate}[label=\textup{(}\textit{\alph*}\textup{)}]
    \item $(x+\xi,y) \notin \rho \implies \bar{q}^X(x,x+\xi) \leq \bar{q}^Y(y,y+\xi)$,
    \label{hyp:a}
    \item $(x,y+\xi) \notin \rho \implies \bar{q}^X(x,x+\xi) \geq \bar{q}^Y(y,y+\xi)$,
    \label{hyp:b}
    \item $(x+\xi,y+\xi) \notin \rho \implies (x+\xi,y+\xi) \notin E \times E$.
    \label{hyp:c}
\end{enumerate}
Then, for each pair of initial conditions $\rb{x^0, y^0} \in \rho$, there exist two CTMCs $X$ and $Y$ on the same probability space, with common state space $E$ and transition rates $q^X$ and $q^Y$, such that $X(0) = x^0$, $Y(0) = y^0$ and
\begin{equation*}
    \P \Brb{\brb{X(t),Y(t)} \in \rho \ \ \forall t \in [0, \tau^X \wedge \tau^Y)} = 1.
\end{equation*}
\end{theorem}
\end{samepage}

\begin{proof}
Let $W = \brb{W(t)}_{t \geq 0}$ be a CTMC with state space $E \times E$ and transition rates $Q^W$ defined in the following way. The only possible transitions from a state $(x,y) \in E \times E$ are those to states of the form $(x+\xi,y)$, $(x,y+\xi)$, or $(x+\xi,y+\xi)$, with $\xi \in \Z^d \setminus \cb{0}$. We distinguish two scenarios: $(x,y) \in \rho$ and $(x,y) \notin \rho$. In the first scenario, depending on whether three conditions relating $x$, $y$ and $\xi$ are satisfied, we consider eight different cases -- one of which is in fact unattainable. For each fixed $\xi$, each case permits zero, one, or two transitions from $(x,y)$, with rates as follows:
\begin{longtable}{r l l}
    \refstepcounter{case}\label{cas:1} \thecase &
    $\begin{cases}
        (x+\xi,y) \in \rho \\
        (x,y+\xi) \in \rho \\
        (x+\xi,y+\xi) \in \rho
    \end{cases}$ & 
    $\begin{array}{l}
        Q^W\brb{(x,y),(x+\xi,y)} = q^X(x,x+\xi), \\
        Q^W\brb{(x,y),(x,y+\xi)} = q^Y(y,y+\xi);
    \end{array}$ \\ \\
    \refstepcounter{case}\label{cas:2} \thecase &
    $\begin{cases}
        (x+\xi,y) \in \rho \\
        (x,y+\xi) \notin \rho \\
        (x+\xi,y+\xi) \in \rho
    \end{cases}$ & 
    $\begin{array}{l}
        Q^W\brb{(x,y),(x+\xi,y)} = q^X(x,x+\xi)-q^Y(y,y+\xi), \\
        Q^W\brb{(x,y),(x+\xi,y+\xi)} = q^Y(y,y+\xi);
    \end{array}$ \\ \\
    \refstepcounter{case}\label{cas:3} \thecase &
    $\begin{cases}
        (x+\xi,y) \in \rho \\
        (x,y+\xi) \notin \rho \\
        (x+\xi,y+\xi) \notin \rho
    \end{cases}$ & 
    $\begin{array}{l}
        Q^W\brb{(x,y),(x+\xi,y)} = q^X(x,x+\xi);
    \end{array}$ \\ \\
    \refstepcounter{case}\label{cas:4} \thecase &
    $\begin{cases}
        (x+\xi,y) \notin \rho \\
        (x,y+\xi) \in \rho \\
        (x+\xi,y+\xi) \in \rho
    \end{cases}$ & 
    $\begin{array}{l}
        Q^W\brb{(x,y),(x,y+\xi)} = q^Y(y,y+\xi)-q^X(x,x+\xi), \\
        Q^W\brb{(x,y),(x+\xi,y+\xi)} = q^X(x,x+\xi);
    \end{array}$ \\ \\
    \refstepcounter{case}\label{cas:5} \thecase &
    $\begin{cases}
        (x+\xi,y) \notin \rho \\
        (x,y+\xi) \in \rho \\
        (x+\xi,y+\xi) \notin \rho
    \end{cases}$ & 
    $\begin{array}{l}
        Q^W\brb{(x,y),(x,y+\xi)} = q^Y(y,y+\xi);
    \end{array}$ \\ \\
    \refstepcounter{case}\label{cas:6} \thecase &
    $\begin{cases}
        (x+\xi,y) \notin \rho \\
        (x,y+\xi) \notin \rho \\
        (x+\xi,y+\xi) \in \rho
    \end{cases}$ & 
    $\begin{array}{l}
        Q^W\brb{(x,y),(x+\xi,y+\xi)} = q^X(x,x+\xi) = q^Y(y,y+\xi);
    \end{array}$ \\ \\
    \refstepcounter{case}\label{cas:7} \thecase &
    $\begin{cases}
        (x+\xi,y) \notin \rho \\
        (x,y+\xi) \notin \rho \\
        (x+\xi,y+\xi) \notin \rho
    \end{cases}$ & 
    $\begin{array}{r}
        \text{\textit{no transition allowed in this case;}}
    \end{array}$ \\ \\
    \refstepcounter{case}\label{cas:8} \thecase &
    $\begin{cases}
        (x+\xi,y) \in \rho \\
        (x,y+\xi) \in \rho \\
        (x+\xi,y+\xi) \notin \rho
    \end{cases}$ & 
    $\begin{array}{l}
        \text{\textit{not possible: by hypothesis \ref{hyp:c}, $x+\xi \notin E$ or $y+\xi \notin E$,}} \\
        \text{\textit{hence either $(x+\xi,y) \notin \rho$ or $(x,y+\xi) \notin \rho$.}}
    \end{array}$
\end{longtable}
Note that hypotheses \ref{hyp:a} and \ref{hyp:b} ensure that the first rates defined in cases \ref{cas:2} and \ref{cas:4} are non-negative and that in case \ref{cas:6} the two alternative definitions coincide. 

In the second scenario, for $(x,y) \notin \rho$ and $\xi \in \Z^d \setminus \cb{0}$, let
\begin{align*}
    Q^W\brb{(x,y),(x+\xi,y)} & = q^X(x,x+\xi) && \text{if } x+\xi \in E, \\
    Q^W\brb{(x,y),(x,y+\xi)} & = q^Y(y,y+\xi) && \text{if } y+\xi \in E,
\end{align*}
and set to zero all other transition rates from state $(x,y)$.

Observe that $Q^W$ is defined precisely so that it has no transitions out of $\rho$. Hence, under the initial condition $W(0) = \rb{x^0, y^0} \in \rho$, one has
\begin{equation*} 
    \P \brb{W(t) \in \rho \ \ \forall t \in [0, \tau^W)} = 1. 
\end{equation*}

What remains to be proved is that $W(t)$ can be written as $\brb{X(t),Y(t)}$ for all $t \in [0, \tau^W)$, where $X$ and $Y$ are CTMCs with rates $q^X$ and $q^Y$ and such that $\tau^X \wedge \tau^Y = \tau^W$. To do so, it suffices to apply Theorem~\ref{thm:app} in Appendix~\ref{sec:app}, a general result on \emph{lumpability} of (possibly explosive) CTMCs. First, note that the absence of instantaneous states for $Q^W$ follows easily from that of $q^X$ and $q^Y$. Let us then check the validity of the hypotheses \eqref{eq:indy} and \eqref{eq:indx}, the verification of the former being presented in detail below, the latter being completely analogous. The claim is that for all $x, y \in E$ and $\xi \in \Z^d \setminus \cb{0}$ with $x+\xi \in E$,
\begin{equation}
\label{eq:claim}
    \sum_{z' \in E} Q^W\brb{(x,y), (x+\xi,z')} = q^X(x,x+\xi).
\end{equation}
Given how $Q^W$ is constructed, the left-hand side of \eqref{eq:claim} can be written as
\begin{equation}
\label{eq:twosummands}
    \sum_{z' \in E} Q^W\brb{(x,y), (x+\xi,z')} = Q^W\brb{(x,y), (x+\xi,y)} + Q^W\brb{(x,y), (x+\xi,y+\xi)},
\end{equation}
where $Q^W\brb{(x,y), (x+\xi,y+\xi)}$ is interpreted as zero whenever $y+\xi \notin E$.

If $(x,y) \in \rho$, depending on the mutual relation of $x$, $y$, and $\xi$, the right-hand side of \eqref{eq:twosummands} becomes
\begin{longtable}{r l}
    \ref{cas:1} &
    $\begin{aligned}
        Q^W\brb{(x,y),(x+\xi,y)} = q^X(x,x+\xi),
    \end{aligned}$ \\ \\
    \ref{cas:2} &
    $\begin{aligned}
        Q^W\brb{(x,y),(x+\xi,y)} + Q^W\brb{(x,y),(x+\xi,y+\xi)} & \\
        = \bsb{q^X(x,x+\xi) - q^Y(y,y+\xi)} + q^Y(y,y+\xi) & = q^X(x,x+\xi),
    \end{aligned}$ \\ \\
    \ref{cas:3} &
    $\begin{aligned}
        Q^W\brb{(x,y),(x+\xi,y)} = q^X(x,x+\xi),
    \end{aligned}$ \\ \\
    \ref{cas:4} &
    $\begin{aligned}
        Q^W\brb{(x,y),(x+\xi,y+\xi)} = q^X(x,x+\xi),
    \end{aligned}$ \\ \\
    \ref{cas:5} &
    $\begin{aligned}
        \textit{this case can be excluded, as shown below},
    \end{aligned}$ \\ \\
    \ref{cas:6} &
    $\begin{aligned}
        Q^W\brb{(x,y),(x+\xi,y+\xi)} = q^X(x,x+\xi),
    \end{aligned}$ \\ \\
    \ref{cas:7} &
    $\begin{aligned}
        0 = q^X(x,x+\xi), \ \textit{as justified below},
    \end{aligned}$
\end{longtable}
where the rationale for excluding \ref{cas:5} is that in this case we know that $(x+\xi,y+\xi) \notin \rho$ and $(x,y+\xi) \in \rho$, so $(x+\xi,y+\xi) \notin E \times E$ -- by hypothesis \ref{hyp:c} of the theorem -- and $y+\xi \in E$, which are incompatible with our assumption that $x+\xi \in E$. As for the equality in \ref{cas:7}, it holds because also in this case we have $(x+\xi,y+\xi) \notin E \times E$ by hypothesis \ref{hyp:c}, so that $y+\xi \notin E$ (again, recall that $x+\xi \in E$) and thus $\bar{q}^Y(y,y+\xi)=0$; finally, remembering that in case \ref{cas:7} we have $(x+\xi,y) \notin \rho$ and applying hypothesis \ref{hyp:a}, we obtain
\begin{equation*}
    0 \leq q^X(x,x+\xi) = \bar{q}^X(x,x+\xi) \leq \bar{q}^Y(y,y+\xi) = 0.
\end{equation*}

On the other hand, if $(x,y) \notin \rho$, since $x+\xi \in E$, the right-hand side of \eqref{eq:twosummands} becomes
\begin{equation*}
    Q^W\brb{(x,y),(x+\xi,y)} = q^X(x,x+\xi).
\end{equation*}

We have just shown that $Q^W$ satisfies condition \eqref{eq:claim} for all admissible choices of $x$, $y$, and $\xi$. Thus, it is possible to apply Theorem~\ref{thm:app} and obtain the existence of the two desired CTMCs $X$ and $Y$ with rates $q^X$ and $q^Y$, such that $\brb{X(0), Y(0)} = \rb{x^0, y^0} \in \rho$ and
\begin{equation*}
    \P \Brb{\brb{X(t),Y(t)} \in \rho \ \ \forall t \in [0, \tau^X \wedge \tau^Y)} = 1.
\end{equation*}
\end{proof}

Let us now see a simple one-dimensional example in which this theorem can be applied to two MAK SRNs with a common underlying reaction network and two different choices of rate constants. They will be compared  by means of a binary relation $\rho$ that is not of the form $\preord{M}$, and is not even a preorder.

\vskip 10 pt

\begin{example}
\label{ex:pcc}
Consider the reaction network and associated MAKs $K^X$ and $K^Y$ depicted in Figure~\ref{fig:one}, giving rise to two CTMCs $X$ and $Y$ counting the molecules of A in the two different regimes. We want to compare $X$ and $Y$ with respect to some binary relation $\rho$ defined on their common state space $\nat$, using Theorem~\ref{thm:pcc}.

\begin{figure}[htbp]
    \begin{center}
        \vskip 8 pt
        \begin{tikzpicture}
            \node[state] (0)    at (0,0)    {$\text{0}$}; 
            \node[state] (A)    at (3,0)    {$\text{A}$}; 
            \node[state] (2A)   at (6,0)    {$2\text{A}$};
            \path[-{Stealth}]
            (0)     edge["{$\kappa_{0 \to \text{A}}^X, \kappa_{0 \to \text{A}}^Y$}"]    node{} (A)
            (A)     edge["{$\kappa_{\text{A} \to 2\text{A}}^X, \kappa_{\text{A} \to 2\text{A}}^Y$}"]     node{} (2A)
            (2A)    edge["{$\kappa_{2\text{A} \to 0}^X, \kappa_{2\text{A} \to 0}^Y$}"][bend left = 30]    node{} (0);
        \end{tikzpicture}
    \end{center}
\caption{a one-dimensional reaction network and two different associated sets of rate constants $K^X$ and $K^Y$.}
\label{fig:one}
\end{figure}
    
The reaction vectors are
\begin{equation*}
    \xi_{0 \to \text{A}} = \xi_{\text{A} \to 2\text{A}} = 1 \ \ \ \ \ \text{and} \ \ \ \ \ \xi_{2\text{A} \to 0} = -2,
\end{equation*}
while the corresponding transition rates and propensities associated with $X$ are
\begin{align*}
    q^X(x, x+1) & = \lambda_{0 \to \text{A}}^X(x) + \lambda_{\text{A} \to 2\text{A}}^X(x) = \kappa_{0 \to \text{A}}^X + \kappa_{\text{A} \to 2\text{A}}^X x \ \ \ \text{and} \\
    q^X(x, x-2) & = \lambda_{2\text{A} \to 0}^X(x) = \kappa_{2\text{A} \to 0}^X x(x-1),
\end{align*}
and similarly for $Y$. Let $\rho \subseteq \nat \times \nat$ be the binary relation
\begin{equation*}
    (x,y) \in \rho \ \ \text{if and only if} \ \ x \leq y+1,
\end{equation*}
which resembles $\preord{1}$, the matrix preorder determined by $1 \in \R^{1,1}$ (i.e.\ the usual order in $\nat$), except for the $+1$ constant that prevents it from being of the form $\preord{M}$ -- and, more generally, from satisfying the transitive property. We cannot simply use $\preord{1}$ due to boundary effects. Heuristically, if the reaction $2\text{A} \to 0$ fires in process $Y$ when $\brb{X(t),Y(t)} = \rb{1,2}$, it is impossible for $X$ to decrease at the same time, hence the state $\rb{1,0}$ is visited and $\preord{1}$ is not preserved. As a matter of fact, it can be verified that $\preord{1}$ fails to satisfy the hypotheses of Theorem~\ref{thm:pcc}. We therefore proceed with $\rho$ and check whether, and for which values of $\rb{\kappa_{0 \to \text{A}}^X, \kappa_{\text{A} \to 2\text{A}}^X, \kappa_{2\text{A} \to 0}^X, \kappa_{0 \to \text{A}}^Y, \kappa_{\text{A} \to 2\text{A}}^Y, \kappa_{2\text{A} \to 0}^Y}$, the objects defined above satisfy conditions \ref{hyp:a} and \ref{hyp:b} of Theorem~\ref{thm:pcc} for all choices of $(x,y) \in \rho$ and steps $\xi$ (the validity of \ref{hyp:c} is natural in this case). Starting from $\xi = 1$:
\begin{itemize}
    \item [\ref{hyp:a}] $(x,y) \in \rho$ and $(x+1,y) \notin \rho$ means that $x \leq y+1 < x+1$, hence $x=y+1$. Under this assumption, $\bar{q}^X(x, x+1) \leq \bar{q}^Y(y, y+1)$ must be true, so
    \begin{equation}
    \label{eq:kon}
        \kappa_{0 \to \text{A}}^X + \kappa_{\text{A} \to 2\text{A}}^X (y+1) \leq \kappa_{0 \to \text{A}}^Y + \kappa_{\text{A} \to 2\text{A}}^Y y \ \ \ \ \forall y \in \nat;
    \end{equation}
    \item [\ref{hyp:b}] $(x,y) \in \rho$ and $(x,y+1) \notin \rho$ means that $y+2 < x \leq y+1$, which is not possible. We do not get additional constraints on the rate constants from here.
\end{itemize}

As for $\xi=-2$, the extension $\bar{q}$ of $q$ gives \ref{hyp:a} for $x<2$ and \ref{hyp:b} for $y<2$. Otherwise:
\begin{itemize}
    \item [\ref{hyp:a}] $(x,y) \in \rho$ and $(x-2,y) \notin \rho$ means that $x \leq y+1 < x-2$, which is not possible;
    \item [\ref{hyp:b}] $(x,y) \in \rho$ and $(x,y-2) \notin \rho$ means that $y-1 < x \leq y+1$, which is satisfied either when $x = y$ or $x = y+1$. In both cases, we must make sure that $\bar{q}^X(x, x-2) \geq \bar{q}^Y(y, y-2)$ is satisfied, obtaining
    \begin{equation}
    \label{eq:ktw}
        \begin{cases}
            \kappa_{2\text{A} \to 0}^X y(y-1) \geq \kappa_{2\text{A} \to 0}^Y y(y-1) \\
            \kappa_{2\text{A} \to 0}^X (y+1)y \geq \kappa_{2\text{A} \to 0}^Y y(y-1)
        \end{cases}
        \forall y \in \nat.
    \end{equation}
\end{itemize}

For any other $\xi \in \Z \setminus \cb{0}$, we have that $\bar{q}^X(x, x+\xi) = \bar{q}^Y(y, y+\xi) = 0$, so hypotheses \ref{hyp:a} and \ref{hyp:b} are automatically satisfied. If we now put \eqref{eq:kon} and \eqref{eq:ktw} together, considering that the second inequality in \eqref{eq:ktw} is weaker than the first one and can be ignored, we are left with
\begin{equation*}
    \begin{cases}
        \kappa_{0 \to \text{A}}^X + \kappa_{\text{A} \to 2\text{A}}^X + \kappa_{\text{A} \to 2\text{A}}^X y \leq \kappa_{0 \to \text{A}}^Y + \kappa_{\text{A} \to 2\text{A}}^Y y \\ 
        \kappa_{2\text{A} \to 0}^X y(y-1) \geq \kappa_{2\text{A} \to 0}^Y y(y-1)
    \end{cases}
    \forall y \in \nat,
\end{equation*}
and such conditions are certainly met when
\begin{equation*}
    \begin{cases}
        \kappa_{0 \to \text{A}}^X + \kappa_{\text{A} \to 2\text{A}}^X \leq \kappa_{0 \to \text{A}}^Y \\
        \kappa_{\text{A} \to 2\text{A}}^X \leq \kappa_{\text{A} \to 2\text{A}}^Y \\
        \kappa_{2\text{A} \to 0}^X \geq \kappa_{2\text{A} \to 0}^Y.
    \end{cases}
\end{equation*}

Under these constraints, by applying Theorem~\ref{thm:pcc}, we know that the corresponding CTMCs $X$ and $Y$ can be constructed on the same probability space, in such a way that
\begin{equation*}
    \P \brb{X(t) \leq Y(t) + 1 \ \ \forall t \geq 0} = 1.
\end{equation*}
Note that in this case the comparison of the two processes is valid for all times, since $X$ and $Y$ are non-explosive. This property holds independently of the values of the rate constants, as all the propensities associated with positive reaction vectors are at most linear in the state variables.

Even though $\rho$ is not a preorder, it still yields a useful bound on $X$ relative to $Y$. In many situations, this bound is as effective as one would obtain using $\preord{1}$. For instance, it allows one to transfer from $Y$ to $X$ the possible tightness of the family of time-$t$ distributions. Example~\ref{ex:erg} illustrates the practical utility of this transfer property.
\end{example}

\section{Linear sufficient conditions for SRNs}

The main drawback of Theorem~\ref{thm:pcc} is that its hypotheses require checking a set of conditions for every possible choice of $x$, $y$, and $\xi$ -- a potentially infinite number of cases. As just shown in Example~\ref{ex:pcc}, there are instances in which these checks are computationally feasible, but in many cases they are not. In what follows, we will focus on the application of Theorem~\ref{thm:pcc} to SRNs with monotone kinetics. Our aim is to derive conditions that are easier to verify, at least computationally. We achieve this by proving Theorem~\ref{thm:lsc}, where the assumptions are formulated as a finite set of linear constraints and the binary relation $\rho$ is replaced by a more tractable matrix preorder. We also provide code in Section~\ref{sub:imp} to check these constraints efficiently. Together, these contributions make our result particularly useful from a practical point of view.

\subsection{Theoretical result}

In the following theorem, we refer to a reaction network $\rn$ with all its components and features as discussed and named in Section~\ref{sub:srn}. In particular, recall that once the functions $f_r$ have been fixed, a monotone kinetics is identified with its collection of rate constants $K = \rb{\kappa_r}_{r \in \rct}$ and that we denote by $C$ a matrix displaying in its rows a basis of the subspace of $\R^d$ spanned by the conservation laws associated with $\rn$. 

\vskip 10 pt

\begin{samepage}
\begin{theorem}
\label{thm:lsc}
    Consider a reaction network $\rn$ with two associated monotone kinetics sharing the functions $f_r$, with rate constants $K^X$ and $K^Y$. Choose $M \in \Z^{m,d}$ and define the matrix
    \begin{equation*}
        D = \brb{M^\intercal | C^\intercal} \in \R^{d,m+d-s}.
    \end{equation*}
    
    Let $A,B \in \R^{m+1,d}$ be filled in as follows, for $i = 1,\dots,m+1$ and $j=1,\dots,d$:
    \begin{equation*}
        A_{i,j} = 
        \begin{cases}
            1 & \text{if} \ \exists \alpha \in \R^{m+d-s} \ \text{such that} \ D\alpha = e_j \ \text{and} \ \alpha_k \geq 0 \ \forall k \in \{1,\dots,m\} \setminus \{i\}, \\
            0 & \text{otherwise};
        \end{cases}
    \end{equation*}
    \begin{equation*}
        B_{i,j} = 
        \begin{cases}
            1 & \text{if} \ \exists \beta \in \R^{m+d-s} \ \text{such that} \ D\beta = e_j \ \text{and} \ \beta_k \leq 0 \ \forall k \in \{1,\dots,m\} \setminus \{i\}, \\
            0 & \text{otherwise}.
        \end{cases}
    \end{equation*}
    
    Now, suppose that each reaction $r \in \rct$ satisfies one of

    \begin{tabular}{r l c r c l c l}
        \refstepcounter{hypa}\label{hyp:a1} \thehypa & $(M\xi_r)^+ = 0$, & & & & & & \\
        \refstepcounter{hypa}\label{hyp:a2} \thehypa & & & $A_{m+1,*} \odot \nu_r = \nu_r$, & & $\kappa_r^X \leq \kappa_r^Y$, & & \\
        \refstepcounter{hypa}\label{hyp:a3} \thehypa & $(M\xi_r)^+ = e_i$, & & $A_{i,*} \odot \nu_r = \nu_r$, & & $\kappa_r^X \leq \kappa_r^Y$ & & for some $i \in \cb{1,\dots,m}$,
    \end{tabular}
    
    and one of
    
    \begin{tabular}{r l c r c l c l}
        \refstepcounter{hypb}\label{hyp:b1} \thehypb & $(M\xi_r)^- = 0$, & & & & & & \\
        \refstepcounter{hypb}\label{hyp:b2} \thehypb & & & $B_{m+1,*} \odot \nu_r = \nu_r$, & & $\kappa_r^X \geq \kappa_r^Y$, & & \\
        \refstepcounter{hypb}\label{hyp:b3} \thehypb & $(M\xi_r)^- = e_i$, & & $B_{i,*} \odot \nu_r = \nu_r$, & & $\kappa_r^X \geq \kappa_r^Y$ & & for some $i \in \cb{1,\dots,m}$.
    \end{tabular}

    Then, for each choice of $x^0, y^0 \in \nat^d$ such that $x^0-y^0 \in \ss$ and $x^0 \preord{M} y^0$, there exist two CTMCs $X$ and $Y$ modeling $\rn$ with kinetics $K^X$ and $K^Y$, defined on a common probability space, and such that $X(0) = x^0$, $Y(0) = y^0$ and
    \begin{equation*}
        \P\brb{X(t) \preord{M} Y(t) \ \ \forall t \in [0, \tau^X \wedge \tau^Y)} = 1.
    \end{equation*}
\end{theorem}
\end{samepage}

\vskip 10 pt

\begin{remark}
\label{rem:dif}
    For mass-action kinetics, even though Theorem~\ref{thm:lsc} is stated for a single reaction network with two choices of rate constants, it also allows us to compare two structurally different SRNs, with reaction sets $\rct_1 \neq \rct_2$. In that case, we consider $\rn$ of Theorem~\ref{thm:lsc} as the network with reactions $\rct = \rct_1 \cup \rct_2$. This approach works because several hypotheses of Theorem~\ref{thm:lsc} involve inequalities between homologous rate constants: since rate constants are non‑negative, any such inequality is automatically satisfied if we set its smaller term to zero. Doing so is essentially equivalent to \emph{removing} the corresponding reaction in model $X$ but not necessarily $Y$ (if the inequality comes from \ref{hyp:a2} or \ref{hyp:a3}), or in model $Y$ but not necessarily $X$ (if the inequality comes from \ref{hyp:b2} or \ref{hyp:b3}). More details on this strategy and its applications can be found in Examples~\ref{ex:lkv} and \ref{ex:erg}.
\end{remark}

\begin{proof}[Proof of Theorem~\ref{thm:lsc}]
    In order to prove this result, for each $x^0 \in \nat^d$ temporarily fixed, we will apply our previous Theorem~\ref{thm:pcc} choosing as
    \begin{itemize}
        \item $E$: the stoichiometric compatibility class $\rb{x^0 + \ss} \cap \nat^d$,
        \item $q^X, q^Y$: the rate matrices defined on $E \times E$ and computed according to the monotone kinetics associated with $K^X$ and $K^Y$ (extended to $\bar{q}^X$ and $\bar{q}^Y$ as in Theorem~\ref{thm:pcc}),
        \item $\rho$: the restriction of the preorder $\preord{M}$ to $E$.
    \end{itemize}
    
    Let $x,y \in E$ be such that $x \preord{M} y$. According to the statement of Theorem~\ref{thm:pcc}, we should verify that for all $\xi \in \Z^d \setminus \cb{0}$
    \begin{itemize}
        \item [\ref{hyp:a}] $x+\xi \npreord{M} y \implies \bar{q}^X(x, x+\xi) \leq \bar{q}^Y(y, y+\xi)$,
        \item [\ref{hyp:b}] $x \npreord{M} y+\xi \implies \bar{q}^X(x, x+\xi) \geq \bar{q}^Y(y, y+\xi)$,
        \item [\ref{hyp:c}] $x+\xi \npreord{M} y+\xi \implies x+\xi \notin E \ \ \text{or} \ \ y+\xi \notin E$.
    \end{itemize}

    Note that \ref{hyp:c} is automatically satisfied due to the definition of $\preord{M}$, independently of $\xi$. Conditions \ref{hyp:a} and \ref{hyp:b} are analogous to those in \cite[Theorem~3.1]{campos2023comparison}. We verify the following stronger, reactionwise conditions for every $r\in\rct$:
    \begin{enumerate}[label=\textup{(}\textit{\alph*}*\textup{)}]
        \item $x+\xi_r \npreord{M} y \implies \lambda_r^X(x) \leq \lambda_r^Y(y)$,
        \label{hyp:a*}
        \item $x \npreord{M} y+\xi_r \implies \lambda_r^X(x) \geq \lambda_r^Y(y)$.
        \label{hyp:b*}
    \end{enumerate}

    If \ref{hyp:a*} and \ref{hyp:b*} are true for each $r$, conditions \ref{hyp:a} and \ref{hyp:b} hold for all $\xi \in \Z^d \setminus \cb{0}$ as well, as a consequence of the identities 
    \begin{equation*}
        \bar{q}^X(x, x+\xi) = \sum\limits_{\substack{r \in \rct \\ \xi_r = \xi}} \lambda_r^X(x), \ \ \ \ \ \ \ \ \ \ \ \ \ \ \ \bar{q}^Y(y, y+\xi) = \sum\limits_{\substack{r \in \rct \\ \xi_r = \xi}} \lambda_r^Y(y).
    \end{equation*}

    Recall now, as introduced in \eqref{eq:mok}, that the monotone rate functions appearing in \ref{hyp:a*} and \ref{hyp:b*} can be written as
    \begin{equation}
    \label{eq:rts}
        \lambda_r^X(x) = f_r(\kappa_r^X, x), \ \ \ \ \ \ \ \ \ \ \ \ \ \ \ \lambda_r^Y(y) = f_r(\kappa_r^Y, y),
    \end{equation}
    with $f_r$ non-decreasing with respect to all of its arguments, i.e.\ the rate constant and the reactant species counts. Let us show that \ref{hyp:a*} is true whenever any of the assumptions \ref{hyp:a1}, \ref{hyp:a2} or \ref{hyp:a3} reported in the theorem statement holds. The verification for \ref{hyp:b*} is completely analogous.
    
    \begin{itemize}
        \item [\ref{hyp:a1}] The hypothesis $(M\xi_r)^+ = 0$ is equivalent to $M\xi_r \leq 0$, so $x+\xi_r \npreord{M} y$ cannot be true while $x \preord{M} y$, because under such assumption
        \begin{equation*}
            M(x+\xi_r-y) = M(x-y) + M\xi_r \leq 0.
        \end{equation*}
        This means that the antecedent of the implication is false and therefore \ref{hyp:a*} is verified.
        
        \item [\ref{hyp:a2}] The hypothesis $A_{m+1,*} \odot \nu_r = \nu_r$ means that $A_{m+1,j} = 1$ for all $j \in \text{supp}(\nu_r)$, that is 
        \begin{equation*}
            \forall j \in \text{supp}(\nu_r) \ \exists \alpha \in \R_{\geq 0}^m \times \R^{d-s} : D\alpha = e_j.
        \end{equation*}
        
        Let us fix $j \in \text{supp}(\nu_r)$ and consider the above mentioned $\alpha$, satisfying
        \begin{equation*}
            \begin{cases}
                \alpha_1 M_{1,*}+\dots+\alpha_m M_{m,*}+\alpha_{m+1} C_{1,*}+\dots+\alpha_{m+d-s} C_{d-s,*} = e_j, \\
                \alpha_1,\dots,\alpha_m \geq 0.
            \end{cases}
        \end{equation*}
        
        Now -- by definition of $x \preord{M} y$ -- we have $M(x-y) \leq 0$, that is
        \begin{equation*}
            \langle M_{k,*}, x-y \rangle \leq 0 \ \ \ k = 1,\dots,m,
        \end{equation*}
        so that
        \begin{equation*}
            \langle e_j, x-y \rangle = \sum_{k=1}^m \underbrace{\alpha_k}_{\geq 0} \underbrace{\langle M_{k,*}, x-y \rangle}_{\leq 0} + \sum_{k=1}^{d-s} \alpha_{m+k} \underbrace{\langle \overbrace{C_{k,*}}^{\in \ss^\perp}, \overbrace{x-y}^{\in \ss} \rangle}_{= 0} \leq 0,
        \end{equation*}
        i.e.\ $x_j \leq y_j$.
        
        This holds for every $j \in \text{supp}(\nu_r)$, so -- together with \eqref{eq:rts} and the fact that $\kappa_r^X \leq \kappa_r^Y$ (the second part of \ref{hyp:a2}) -- it implies
        \begin{equation*}
            \lambda_r^X(x) \leq \lambda_r^Y(y),
        \end{equation*}
        and \ref{hyp:a*} is satisfied. 
        
        \item [\ref{hyp:a3}] The hypothesis $A_{i,*} \odot \nu_r = \nu_r$ means that $A_{i,j} = 1$ for all $j \in \text{supp}(\nu_r)$, that is 
        \begin{equation*}
            \forall j \in \text{supp}(\nu_r) \ \exists \alpha \in \R_{\geq 0}^{i-1} \times \R \times \R_{\geq 0}^{m-i} \times \R^{d-s} : D\alpha = e_j.
        \end{equation*}
        
        As in the previous case, fix $j \in \text{supp}(\nu_r)$ and consider the above mentioned $\alpha$, satisfying
        \begin{equation*}
            \begin{cases}
                \alpha_1 M_{1,*}+\dots+\alpha_m M_{m,*}+\alpha_{m+1} C_{1,*}+\dots+\alpha_{m+d-s} C_{d-s,*} = e_j, \\
                \alpha_1,\dots,\alpha_{i-1},\alpha_{i+1},\dots,\alpha_m \geq 0.
            \end{cases}
        \end{equation*}
        
        Since in this case the sign of $\alpha_i$ is unknown, we need $\langle M_{i,*}, x-y \rangle$ to vanish if we want to prove $x_j \leq y_j$ in a similar manner as before. In order to show that $\langle M_{i,*}, x-y \rangle = 0$, we will exploit the fact that
        \begin{equation*}
            \begin{cases}
                x \preord{M} y, \\
                x+\xi_r \npreord{M} y,
            \end{cases}
        \end{equation*}
        and use the assumption $\rb{M\xi_r}^+ = e_i$ of \ref{hyp:a3}. By definition, $x+\xi_r \npreord{M} y$ means that $M(x+\xi_r-y) \nleq 0$, so there exists some $k \in \{1,\dots,m\}$ such that
        \begin{equation}
        \label{eq:npo}
            0 < \langle M_{k,*}, x+\xi_r-y \rangle = \langle M_{k,*}, x-y \rangle + \langle M_{k,*}, \xi_r \rangle.
        \end{equation}
        This is not possible when $k \neq i$, because for such values of $k$ 
        \begin{equation*}
            \langle M_{k,*}, \xi_r \rangle \leq {\langle M_{k,*}, \xi_r \rangle}^+ = 0,
        \end{equation*}
        and this, combined with \eqref{eq:npo}, would contradict the fact that $x \preord{M} y$. Hence, it must be true when $k=i$, in which case
        \begin{equation*}
            \langle M_{i,*}, \xi_r \rangle = {\langle M_{i,*}, \xi_r \rangle}^+ = 1,
        \end{equation*}
        and so from \eqref{eq:npo} we get
        \begin{equation*}
            -1 < \langle M_{i,*}, x-y \rangle.
        \end{equation*}
        
        We also know that
        \begin{description} 
            \item[] $\langle M_{i,*}, x-y \rangle \leq 0$ because $x \preord{M} y$,
            \item[] $\langle M_{i,*}, x-y \rangle \in \Z$ because both $M_{i,*}$ and $x-y$ belong to $\Z^d$,
        \end{description}
        so we can conclude that
        \begin{equation*}
            \langle M_{i,*}, x-y \rangle = 0,
        \end{equation*}
        as desired. We are finally able to say that
        \begin{equation*}
            \langle e_j, x-y \rangle = \sum\limits_{\substack{k=1 \\ k \neq i}}^m \underbrace{\alpha_k}_{\geq 0} \underbrace{\langle M_{k,*}, x-y \rangle}_{\leq 0} + \alpha_i \underbrace{\langle M_{i,*}, x-y \rangle}_{=0} + \sum_{k=1}^{d-s} \alpha_{m+k} \underbrace{\langle \overbrace{C_{k,*}}^{\in \ss^\perp}, \overbrace{x-y}^{\in \ss} \rangle}_{=0} \leq 0.
        \end{equation*}
        
        As in the \ref{hyp:a2} case, this -- together with the remaining hypothesis $\kappa_r^X \leq \kappa_r^Y$ and the monotonicity of $f_r$ -- is what we needed to guarantee the validity of \ref{hyp:a*}.
    \end{itemize}
    
    The hypotheses of Theorem~\ref{thm:pcc} are therefore all verified: we can apply it once for every stoichiometric compatibility class (determined by $x^0$) associated with $\rn$ and obtain the conclusion.
\end{proof}

\vskip 10 pt

\begin{remark}
\label{rem:matrix}
    As anticipated in Section~\ref{sub:not}, multiple matrices can induce the same preorder. In addition, when applying Theorem~\ref{thm:lsc} to a reaction network $\rn$, we deduce the existence of two comparable CTMCs lying in a fixed stoichiometric compatibility class $\rb{x^0 + \ss} \cap \nat^d$, so we are only interested in the meaning of the preorder $\preord{M}$ when restricted to $\rb{x^0 + \ss} \cap \nat^d$. Consequently, as long as $M' \in \Z^{m',d}$ is such that $M' \sim_{\ss} M$ (and hence $M' \sim_{\rb{x^0 + \ss} \cap \nat^d} M$ for all choices of $x^0$, see definitions \eqref{eq:preord} and \eqref{eq:equiv}), we are free to replace $M$ with $M'$ without altering the thesis of Theorem~\ref{thm:lsc}. However, notice that an alternative choice of the matrix $M$ can surely affect hypotheses \ref{hyp:a3} and \ref{hyp:b3} -- potentially invalidating them when they were previously true or vice versa -- so it is of fundamental importance to carefully select the right representative of the equivalence class of matrices associated with the desired preorder. Here are two recommendations on how to choose the matrix $M$:
    \begin{enumerate}
        \item $\text{GCD}(M_{i,*}) \in \cb{0,1}$ for each $i = 1,\dots,m$, that is, every non-zero row of $M$ must be made of coprime numbers. The reason for this is that the $i$-th component of $M \xi_r$ -- which we would possibly want to be either $0$ or $\pm 1$, to have a chance of \ref{hyp:a3} or \ref{hyp:b3} being verified -- is necessarily a multiple of $\text{GCD}(M_{i,*})$. Requiring the matrix $M$ to meet such condition is not restricting whatsoever, since each non-zero row of any candidate $M$ can always be divided by its $\text{GCD}$ for this purpose, without any effect on $\preord{M}$;
        \item For the same reason described above (making it possible for $(M\xi_r)^+$ and $(M\xi_r)^-$ to be made of all $0$ except for a $1$), $M$ should not contain any ``extra information''. More precisely, if one of the rows of $M$ can be written as a positive linear combination of the others up to a conservation law (i.e.\ up to a linear combination of the rows of $C$), it is not really modifying the preorder $\preord{M}$ on the stoichiometric compatibility classes, but, on the other hand, it could prevent \ref{hyp:a3} or \ref{hyp:b3} from being true. Suppose for example that $M_{1,*}\xi_r = 1$, $M_{2,*}\xi_r = 0$ and $M_{3,*} = M_{1,*} + 4 M_{2,*}$: then $M_{3,*}\xi_r = 1 + 4 \cdot 0 = 1$, so $\rb{M\xi_r}^+ = \rb{1,0,1,\dots}$ cannot satisfy \ref{hyp:a3}. At the same time, if $x$ and $y$ are such that $M_{1,*}(x-y) \leq 0$ and $M_{2,*}(x-y) \leq 0$, then also $M_{3,*}(x-y) \leq 0$ is automatically true. This means that the row $M_{3,*}$ is redundant in the definition of $\preord{M}$ and can (in this case should) be removed.
    \end{enumerate}
\end{remark}

\subsection{Implementation}
\label{sub:imp}

The hypotheses of Theorem~\ref{thm:lsc} are algorithmic by nature, and can be easily translated into code performing an efficient check for a given reaction network $\rn$ and an integer-valued matrix $M$, returning as output
\begin{itemize}
    \item whether Theorem~\ref{thm:lsc} certifies the possibility of associating two different monotone kinetics $K^X$ and $K^Y$ to $\rn$ and building the corresponding Markov chains $X$ and $Y$ so that they will be ordered with respect to $\preord{M}$ for all times (up to the first explosion),
    \item if so, under which conditions on $K^X$ and $K^Y$ -- in the form of a set of equalities and inequalities between homologous rate constants.
\end{itemize}
Such a program is by itself a valuable tool, making \cite[Theorem~3.1]{campos2023comparison} now readily applicable for monotone kinetics. However, its real strength is that, by automating and accelerating the check of all hypotheses, it enables someone interested in studying a given reaction network $\rn$ to repeat the checks for every ``simple enough'' matrix, without having to choose one \textit{a priori}. By doing so, the output becomes
\begin{itemize}
    \item the set of ``simple'' matrix preorders for which Theorem~\ref{thm:lsc} certifies a comparison of $X$ and $Y$ under suitable conditions on $K^X$ and $K^Y$,
    \item for each preorder, the corresponding conditions on the rate constants.
\end{itemize}

What has just been described are the input and output of the two main functions that can be found in the R code that we wrote and published in the repository at the following link:

\href{https://github.com/giulio-cuniberti/comparing-srns}{\ttfamily https://github.com/giulio-cuniberti/comparing-srns}

The code was written relying on existing libraries for Linear Programming (LP). That was possible because the solution of linear systems subject to inequalities as described in the statement of Theorem~\ref{thm:lsc} (precisely in the definition of the matrices $A$ and $B$) can be interpreted as a feasibility problem in the LP framework. This is very convenient: algorithms for LP problems have been designed and optimized for decades and are extremely efficient and fast. Notice, in particular, that the computation of the entries of $A$ and $B$ requires solving multiple linear systems that all share the same coefficient matrix (for a fixed $M$). Moreover, the solution set of each one of these systems is then intersected with $m+1$ different systems of inequalities to compute $m+1$ distinct entries. The implementation uses \texttt{rcdd}/\texttt{cddlib} to obtain all entries in each required row of $A$ and $B$ from a single description, by linear equalities and inequalities, of the corresponding polyhedral cone. If this computation fails, the corresponding LP feasibility problems are solved individually.

Regarding the above description of the second main function of our code -- searching for suitable preorders for a given reaction network -- we specify that we chose to consider as ``simple'' matrices (and associated preorders) those whose set of rows is contained in $\cb{e_1,\dots,e_d,-e_1,\dots,-e_d}$. This results in all species counts being compared separately from one another between the two processes: some will be larger in $X$, some in $Y$, some will be equal (when both $e_j$ and $-e_j$ are present in $M$ for the corresponding $j$) and some others will not be compared. All possible combinations of this kind are explored by our code. Once again, let us emphasize the inherent parallelism of the algorithm: for every considered matrix, the corresponding computations and checks to be performed are similar but mutually independent, so that the search can be carried out along multiple parallel branches. When many computational units are available, this can greatly accelerate the execution of the code.

The conservation laws are computed using exact integer and rational arithmetic, as proposed in \cite{sauro2004conservation}. Exact arithmetic is also used to compute the matrices $A$ and $B$. In particular, the linear systems subject to inequalities are treated using \texttt{rcdd}/\texttt{cddlib}, so that no numerical tolerance is needed to check the conditions. To give an idea of the time required to search for suitable preorders, we applied the algorithm to a collection of reaction networks obtained from BioModels (see \cite{malik2020biomodels}) with at most 10 species and monotone rates. Of these 205 networks, 106 have only mass-action rates, 64 include Michaelis--Menten rates but no Hill rates, and 35 include Hill rates. We do not check whether the models are non-explosive, since Theorem~\ref{thm:lsc} guarantees the comparison up to the first explosion. For a network with $d$ species, there are $(4^d+2^d)/2$ choices of the rows of $M$ to consider, counting only once choices obtained from one another by reversing all comparisons simultaneously. The search considers all these choices, even when a suitable preorder has already been found. In total, $3\,243\,525$ matrices were checked in $1$ hour and $49$ minutes, on an Apple M3 MacBook Air with 16 GB RAM, using one worker for $d\leq 5$ and seven workers for $d\geq 6$. Table~\ref{tab:runtime} reports the computation times and the number of successful searches for each value of $d$, with at least one suitable preorder found for 133 of the 205 networks overall. These results show that the search can be completed in reasonable time for the networks considered, although the number of matrices to check grows exponentially with the number of species. The code needed to reproduce these computations is available in the repository.

\setcounter{table}{0}
\begin{table}[htbp]
\centering
\begin{tabular}{rrrrrr}
\hline
$d$ & Networks & Successful searches & Matrices per network & Median (s) & Maximum (s) \\
\hline
1 & 15 & 15 & 3        & 0.001    & 0.043 \\ 
2 & 26 & 20 & 10       & 0.004    & 0.006 \\ 
3 & 43 & 22 & 36       & 0.026    & 0.093 \\ 
4 & 27 & 19 & 136      & 0.155    & 0.174 \\ 
5 & 19 & 11 & 528      & 0.790    & 1.024 \\ 
6 & 24 & 15 & 2\,080   & 1.280    & 1.850 \\ 
7 & 15 & 10 & 8\,256   & 6.813    & 8.620 \\ 
8 & 21 & 11 & 32\,896  & 44.801   & 61.087 \\ 
9 & 14 & 9  & 131\,328 & 300.455  & 351.731 \\ 
10 & 1 & 1  & 524\,800 & 1332.590 & 1332.590 \\
\hline
\end{tabular}
\caption{measured search time for a single network, depending on the number of species. The reported times cover the checks of all candidate matrices, excluding preliminary setup and final processing. For each $d$, the median and maximum are taken over the networks with $d$ species. The third column reports the number of networks for which the algorithm found at least one non-trivial preorder satisfying the sufficient conditions of Theorem~\ref{thm:lsc}. By trivial preorder we mean the one in which all species are kept equal and the one in which no species is compared; they are always possible.}
\label{tab:runtime}
\end{table}

\subsection{Examples}

Here are several examples of reaction networks $\rn$ for which it is possible to apply Theorem~\ref{thm:lsc} and build two comparable MAK SRNs both based on $\rn$ but with different choices of rate constants. In all examples, assume $C x^0 = C y^0$ and the displayed initial comparisons, as in Theorem~\ref{thm:lsc}. From an intuitive and practical standpoint, this construction tells us how species counts are most likely to change when modifying certain parameters of the MAK model. The examples were found by running our code on a selection of classical reaction networks, frequently studied and discussed in the literature. Interestingly, several networks admit multiple different \textbf{preordering structures} -- i.e.\ systems of rate-constant inequalities that allow the almost-sure preservation of a specified preorder on the species counts. Note that each structure we will present has a trivial symmetric counterpart in which all rate-constant and species-count inequalities are reversed: those will be ignored, since they provide no additional understanding of the considered reaction network. Moreover, we disregard the two trivial comparisons in which the two processes are identical or no species counts are compared, both of which are always possible. To express at a glance the rate-constant and species-count inequalities in our examples, the following color scheme will be used:
\begin{itemize}
    \item \textbf{green} for a \textit{reaction} whose rate constant in model $Y$ is assumed to be greater than or equal to the one in model $X$, or for a \textit{species} whose number of molecules in model $Y$ is greater than or equal to that in model $X$ (at time $0$ by hypothesis, and thereafter by Theorem~\ref{thm:lsc});
    \item \textbf{red} for a \textit{reaction} whose rate constant in model $Y$ is assumed to be less than or equal to the one in model $X$, or for a \textit{species} whose number of molecules in model $Y$ is less than or equal to that in model $X$ (at time $0$ by hypothesis, and thereafter by Theorem~\ref{thm:lsc});
    \item \textbf{blue} for a \textit{reaction} whose rate constants in models $X$ and $Y$ are assumed to be equal, or for a \textit{species} whose number of molecules in models $X$ and $Y$ is equal (at time $0$ by hypothesis, and thereafter by Theorem~\ref{thm:lsc});
    \item \textbf{gray} for a \textit{reaction} whose rate constants in models $X$ and $Y$ are not compared by assumption, or for a \textit{species} whose number of molecules in models $X$ and $Y$ is not compared (neither at time $0$ nor at any other time).
\end{itemize}

Heuristically, the examples below address the question: ``what happens if we start with a MAK SRN and modify its rate constants according to the colors of the reactions?''. Roughly speaking, the answer can be read in the colors of the species: green species will increase, red species will decrease, and blue species will remain the same. We have no control over the behavior of the gray species, and their amount can be freely changed at any time without invalidating the result, as long as the stoichiometric compatibility class does not change.

As was the case in Example~\ref{ex:pcc}, all MAK SRNs presented in this section are non-explosive. The underlying reason is the same, now generalized to networks of arbitrary dimension: the rates of the reactions that cause a net increase in the total molecular count are at most linear in the state variables. To work out the details, one can consider, for instance, the process that counts the total number of molecules of the system at time $t$ and dominate it with a pure birth process with constant jump sizes and a linear birth rate.

\vskip 10 pt

\begin{example}[Reversible reaction]
\label{ex:rev}
    This is one of the simplest examples of a MAK SRN, describing a substrate S being reversibly converted into a product P. The result presented here is somewhat self-evident, but it is important to state it explicitly since it will be referred to in the subsequent Example~\ref{ex:rmm}.
    
    \begin{figure}[htbp]
        \begin{center}
            \vskip 8 pt
            \begin{tikzpicture}
                \node[state] (S)  at (0,0)    {$\textcolor{red}{\text{S}}$}; 
                \node[state] (P)  at (3,0)    {$\textcolor{Green}{\text{P}}$}; 
                \path[-{Stealth[harpoon]}] 
                ([yshift = 2 px]S.east)     edge[color = Green]    node{} ([yshift = 2 px]P.west)
                ([yshift = -2 px]P.west)    edge[color = red]    node{} ([yshift = -2 px]S.east);
            \end{tikzpicture}
        \end{center}
    \caption{the rate constants satisfy $\kappa_{\textup{S} \to \textup{P}}^X \leq \kappa_{\textup{S} \to \textup{P}}^Y$ and $\kappa_{\textup{P} \to \textup{S}}^X \geq \kappa_{\textup{P} \to \textup{S}}^Y$, and by Theorem~\ref{thm:lsc} the species counts fulfill $X_\textup{S}(t) \geq Y_\textup{S}(t)$ and $X_\textup{P}(t) \leq Y_\textup{P}(t)$ for all $t \geq 0$ almost surely.}
    \label{fig:rev}
    \end{figure}
    
    As displayed in Figure~\ref{fig:rev}, in this instance Theorem~\ref{thm:lsc} specializes as follows: if the two MAKs chosen for the reaction network $\text{S} \rightleftharpoons \text{P}$ are such that the forward reaction $\text{S} \to \text{P}$ is stronger in model $Y$ (the associated rate constant is greater than the homologous constant for model $X$) and the reverse reaction $\text{P} \to \text{S}$ is stronger in model $X$, and we let the trajectories of the processes start at points $X(0) = x^0$ and $Y(0) = y^0$ with $x_\text{S}^0 + x_\text{P}^0 = y_\text{S}^0 + y_\text{P}^0$, chosen in such a way that $x^0 \preord{M} y^0$ with $M = \begin{pmatrix} -1 & 0 \\ 0 & 1 \end{pmatrix}$, i.e.\
    \begin{equation*}
        x_\text{S}^0 \geq y_\text{S}^0 \ \ \text{and} \ \ x_\text{P}^0 \leq y_\text{P}^0,
    \end{equation*}
    then it is possible to construct $X$ and $Y$ almost surely preserving $\preord{M}$ for all $t \geq 0$, i.e.\
    \begin{equation*}
        \P\brb{X_\text{S}(t) \geq Y_\text{S}(t), \ X_\text{P}(t) \leq Y_\text{P}(t) \ \ \forall t \geq 0} = 1.
    \end{equation*}
    
    Here we can appreciate the importance of Remark~\ref{rem:matrix}, as Theorem~\ref{thm:lsc} cannot be applied directly to $M$. Instead, we use the equivalent matrix $M' = \begin{pmatrix} 0 & 1 \end{pmatrix}$: if $x$ and $y$ are in the same stoichiometric compatibility class, $x_\text{P} \leq y_\text{P}$ also gives $x_\text{S} \geq y_\text{S}$. In other words, the first row of $M$ is redundant modulo the conservation law $C = \begin{pmatrix} 1 & 1 \end{pmatrix}$ and can be removed. By performing the computations with $M'$, one obtains
    \begin{equation*}
        D = \begin{pmatrix} 0 & 1 \\ 1 & 1 \end{pmatrix}, \qquad
        A = \begin{pmatrix} 1 & 1 \\ 0 & 1 \end{pmatrix}, \qquad
        B = \begin{pmatrix} 1 & 1 \\ 1 & 0 \end{pmatrix}. 
    \end{equation*}
    Indeed, $D \alpha = e_1$ gives $\alpha = \rb{-1,1}$ and $D \alpha = e_2$ gives $\alpha = \rb{1,0}$; the sign restrictions in the definitions of $A$ and $B$ give the displayed entries. Since $M' \xi_{\textup{S} \to \textup{P}} = 1$ and $M' \xi_{\textup{P} \to \textup{S}} = -1$, the forward reaction satisfies \ref{hyp:a3}, \ref{hyp:b1}, and the reverse reaction satisfies \ref{hyp:a1}, \ref{hyp:b3}, under the stated rate inequalities.
    
    Loosely speaking, under our hypotheses on the MAKs, if at time $0$ the process $Y$ has fewer S molecules and more P molecules than $X$, the situation will remain the same for all subsequent times.
\end{example}

\vskip 10 pt

\begin{example}[SIR and SIS]
\label{ex:sis}
    A canonical example of a MAK SRN is the \textbf{SIR} model, which has been used in epidemiology since the beginning of the 20th century (see \cite{kermack1927contribution} for an early deterministic formulation) to predict the spread of infectious diseases. Each member of a population is assigned to one of the three compartments S (Susceptible), I (Infectious) or R (Recovered), and can change its status according to one of two reactions representing the infection and recovery mechanisms. For this model, the algorithm finds no non-trivial preorder in the searched family allowing the application of Theorem~\ref{thm:lsc} (see Figure~\ref{fig:sir}).

    \begin{figure}[htbp]
        \begin{center}
            \vskip 8 pt
            \begin{tikzpicture}
                \node[state] (S+I)  at (0,0)   {$\text{S}+\text{I}$}; 
                \node[state] (2I)   at (2,0)    {$2\text{I}$}; 
                \node[state] (I)    at (4,0)    {$\text{I}$};
                \node[state] (R)    at (6,0)    {$\text{R}$};
                \path[-{Stealth}]
                (S+I)   edge    node{} (2I)
                (I)     edge    node{} (R);
            \end{tikzpicture}
        \end{center}
    \caption{no non-trivial preordering structure can be found with our algorithm for this reaction network.}
    \label{fig:sir}
    \end{figure}
       
    However, our theory can be applied to the closely related \textbf{SIS} model, where recovered individuals immediately return to the susceptible state (acquiring no lasting immunity). The SIS reaction network is also very similar to $\text{S} \rightleftharpoons \text{I}$ (the same as in Example~\ref{ex:rev}, up to renaming P as I), the only difference being that in this case species I autocatalyzes its own production from S. Unsurprisingly, we therefore obtain the same result as in Example~\ref{ex:rev}: strengthening the conversion of S into I and weakening the conversion of I into S leads to fewer S individuals and more I individuals (see Figure~\ref{fig:sis}).

    \begin{figure}[htbp] 
        \begin{center}
            \vskip 8 pt
            \begin{tikzpicture}
                \node[state] (S+I)  at (0,0)    {$\textcolor{red}{\text{S}}+\textcolor{Green}{\text{I}}$}; 
                \node[state] (2I)   at (2,0)    {$2\textcolor{Green}{\text{I}}$}; 
                \node[state] (I)    at (4,0)    {$\textcolor{Green}{\text{I}}$};
                \node[state] (S)    at (6,0)    {$\textcolor{red}{\text{S}}$};
                \path[-{Stealth}]
                (S+I)   edge[color = Green]    node{} (2I)
                (I)     edge[color = red]    node{} (S);
            \end{tikzpicture}
        \end{center}
    \caption{the rate constants satisfy $\kappa_{\textup{S}+\textup{I} \to 2\textup{I}}^X \leq \kappa_{\textup{S}+\textup{I} \to 2\textup{I}}^Y$ and $\kappa_{\textup{I} \to \textup{S}}^X \geq \kappa_{\textup{I} \to \textup{S}}^Y$, and by Theorem~\ref{thm:lsc} the species counts fulfill $X_\textup{S}(t) \geq Y_\textup{S}(t)$ and $X_\textup{I}(t) \leq Y_\textup{I}(t)$ for all $t \geq 0$ almost surely.}
    \label{fig:sis}
    \end{figure}

    The reason why the hypotheses of Theorem~\ref{thm:lsc} are met by SIS but not by SIR is strictly related to the different conservation laws of the two systems. To understand what can go wrong with the SIR model, let $\kappa_{\text{S}+\text{I} \to 2\text{I}}^X \leq \kappa_{\text{S}+\text{I} \to 2\text{I}}^Y$ and $\kappa_{\text{I} \to \text{R}}^X \geq \kappa_{\text{I} \to \text{R}}^Y$ (analogously to the assumptions on the SIS MAKs in Figure~\ref{fig:sis}) and imagine that at some time $t$ one has $\brb{X_{\text{S}}(t),X_{\text{I}}(t),X_{\text{R}}(t)} = \rb{2,1,0}$ and $\brb{Y_{\text{S}}(t),Y_{\text{I}}(t),Y_{\text{R}}(t)} = \rb{1,1,1}$, two points lying in the same stoichiometric compatibility class. $X$ has not fewer S molecules and not more I molecules than $Y$ -- as in the preorder compatible with the SIS reaction network --  but this would no longer be true should $\text{S}+\text{I} \to 2\text{I}$ fire in the $X$ model. However, this cannot be avoided by means of the results presented in this paper, because condition \ref{hyp:a} of Theorem~\ref{thm:pcc} (and therefore also hypotheses \ref{hyp:a1}, \ref{hyp:a2}, \ref{hyp:a3} of Theorem~\ref{thm:lsc}, which are stronger) is not met:
    \begin{equation*}
        \lambda_{\text{S}+\text{I} \to 2\text{I}}^X \brb{(2,1,0)} = \kappa_{\text{S}+\text{I} \to 2\text{I}}^X \cdot 2 \leq \kappa_{\text{S}+\text{I} \to 2\text{I}}^Y = \lambda_{\text{S}+\text{I} \to 2\text{I}}^Y \brb{(1,1,1)}
    \end{equation*}
    is not necessarily true. On the other hand, SIS does not have the same problem, essentially because its states $\rb{2,1}$ and $\rb{1,1}$ are not in the same stoichiometric compatibility class. In order for $\rb{x_\text{S},x_\text{I}}$ and $\rb{y_\text{S},y_\text{I}}$ to belong to the same class (our theory is formulated for this setting), $x_\text{S} + x_\text{I} = y_\text{S} + y_\text{I}$ must hold. Under this assumption, if at some time $t$ the processes $\brb{X_{\text{S}}(t),X_{\text{I}}(t)} = \rb{x_\text{S},x_\text{I}}$ and $\brb{Y_{\text{S}}(t),Y_{\text{I}}(t)} = \rb{y_\text{S},y_\text{I}}$ are related by the preorder described in Figure~\ref{fig:sis}, the relation can be violated as a consequence of $\text{S}+\text{I} \to 2\text{I}$ firing in the $X$ model only if $\rb{x_\text{S},x_\text{I}} = \rb{y_\text{S},y_\text{I}}$, in which case
    \begin{equation*}
        \lambda_{\text{S}+\text{I} \to 2\text{I}}^X \brb{(x_\text{S},x_\text{I})} = \kappa_{\text{S}+\text{I} \to 2\text{I}}^X \cdot  x_\text{S} x_\text{I} \leq \kappa_{\text{S}+\text{I} \to 2\text{I}}^Y \cdot y_\text{S} y_\text{I} = \lambda_{\text{S}+\text{I} \to 2\text{I}}^Y \brb{(y_\text{S},y_\text{I})}
    \end{equation*}
    is always satisfied.
\end{example}

\vskip 10 pt

\begin{example}[Michaelis--Menten]
\label{ex:mm2}
    We will present two different results for the Michaelis--Menten mechanism, already introduced in Example~\ref{ex:mm1}. They also involve some rate constants that are equal in processes $X$ and $Y$ and species counts that are not compared, a feature that we have not yet observed in the previous examples. Even though some species counts cannot be controlled in general, we can still derive partial information on them -- namely when $\brb{X(t),Y(t)}$ lies on the boundary of the chosen preorder relation -- thanks to the conservation laws: this is why these species can nevertheless interact with others whose number of molecules is compared between $X$ and $Y$. This point will be discussed in more detail in Example~\ref{ex:his}.

    The first preorder found by our algorithm coincides with that of \cite[Example~4.1]{campos2023comparison}, but is obtained under weaker conditions on the rate constants: by accelerating the binding of the substrate S to the enzyme E and the subsequent unbinding of the intermediate complex C into E and the product P, and by decelerating its alternative unbinding into S and E, more product and less substrate are expected. For a graphical representation of this scenario and the list of precise inequalities, see Figure~\ref{fig:mm1}.

    \begin{figure}[htbp]
        \begin{center}
            \vskip 8 pt
            \begin{tikzpicture}
                \node[state] (S+E)  at (0,0)    {$\textcolor{red}{\text{S}}+\textcolor{Gray}{\text{E}}$}; 
                \node[state] (C)    at (2,0)    {$\textcolor{Gray}{\text{C}}$}; 
                \node[state] (E+P)  at (4,0)    {$\textcolor{Gray}{\text{E}}+\textcolor{Green}{\text{P}}$};
                \path[-{Stealth[harpoon]}] 
                ([yshift = 2 px]S+E.east)  edge[color = Green]    node{} ([yshift = 2 px]C.west)
                ([yshift = -2 px]C.west)   edge[color = red]    node{} ([yshift = -2 px]S+E.east);
                \path[-{Stealth}]
                (C)       edge[color = Green]                    node{} (E+P);
            \end{tikzpicture}
        \end{center}
    \caption{the rate constants satisfy $\kappa_{\textup{S}+\textup{E} \to \textup{C}}^X \leq \kappa_{\textup{S}+\textup{E} \to \textup{C}}^Y$, $\kappa_{\textup{C} \to \textup{S}+\textup{E}}^X \geq \kappa_{\textup{C} \to \textup{S}+\textup{E}}^Y$ and $\kappa_{\textup{C} \to \textup{E}+\textup{P}}^X \leq \kappa_{\textup{C} \to \textup{E}+\textup{P}}^Y$, and by Theorem~\ref{thm:lsc} the species counts fulfill $X_\textup{S}(t) \geq Y_\textup{S}(t)$ and $X_\textup{P}(t) \leq Y_\textup{P}(t)$ for all $t \geq 0$ almost surely.}
    \label{fig:mm1}
    \end{figure}
    
    On the other hand, if one adopts the same conditions on the MAKs as proposed in \cite{campos2023comparison} (the rate constants of the reactions $\text{S} + \text{E} \rightleftharpoons \text{C}$ must be the same in models $X$ and $Y$, only the one associated with $\text{C} \to \text{E} + \text{P}$ can be greater in $Y$ than in $X$), it is possible to get more precise information, as displayed in Figure~\ref{fig:mm2}: in this case, we can also control the number of molecules of E (that will increase) and C (that will decrease). Hence, our algorithm not only recovered the same result presented in \cite{campos2023comparison} but also went further, showing that it allows for either generalization or refinement.

    \begin{figure}[htbp]
        \begin{center}
            \vskip 8 pt
            \begin{tikzpicture}
                \node[state] (S+E)  at (0,0)    {$\textcolor{red}{\text{S}}+\textcolor{Green}{\text{E}}$}; 
                \node[state] (C)    at (2,0)    {$\textcolor{red}{\text{C}}$}; 
                \node[state] (E+P)  at (4,0)    {$\textcolor{Green}{\text{E}}+\textcolor{Green}{\text{P}}$};
                \path[-{Stealth[harpoon]}] 
                ([yshift = 2 px]S+E.east)  edge[color = blue]    node{} ([yshift = 2 px]C.west)
                ([yshift = -2 px]C.west)   edge[color = blue]    node{} ([yshift = -2 px]S+E.east);
                \path[-{Stealth}]
                (C)       edge[color = Green]                    node{} (E+P);
            \end{tikzpicture}
        \end{center}
    \caption{the rate constants satisfy $\kappa_{\textup{S}+\textup{E} \to \textup{C}}^X = \kappa_{\textup{S}+\textup{E} \to \textup{C}}^Y$, $\kappa_{\textup{C} \to \textup{S}+\textup{E}}^X = \kappa_{\textup{C} \to \textup{S}+\textup{E}}^Y$ and $\kappa_{\textup{C} \to \textup{E}+\textup{P}}^X \leq \kappa_{\textup{C} \to \textup{E}+\textup{P}}^Y$, and by Theorem~\ref{thm:lsc} the species counts fulfill $X_\textup{S}(t) \geq Y_\textup{S}(t)$, $X_\textup{E}(t) \leq Y_\textup{E}(t)$, $X_\textup{C}(t) \geq Y_\textup{C}(t)$ and $X_\textup{P}(t) \leq Y_\textup{P}(t)$ for all $t \geq 0$ almost surely.}
    \label{fig:mm2}
    \end{figure}    
\end{example}

\vskip 10 pt

\begin{example}[Reversible Michaelis--Menten]
\label{ex:rmm}
    In the previous examples, we were able to find preordering structures for both $\text{S} \rightleftharpoons \text{P}$ (see Example~\ref{ex:rev}) and Michaelis--Menten (Example~\ref{ex:mm2}), which is a more detailed description of how $\text{S} \to \text{P}$ can actually occur. One might then expect to still get some results after replacing the reaction $\text{S} \to \text{P}$ with $\text{S} + \text{E} \rightleftharpoons \text{C} \to \text{E} + \text{P}$ in the first network (see Figure~\ref{fig:rmm}), but this is not the case for our test. The problem here is once again connected to the conservation laws. The application of Theorem~\ref{thm:lsc} to $\text{S} \rightleftharpoons \text{P}$ presented in Example~\ref{ex:rev} relies on the fact that, on a given stoichiometric compatibility class, the total number of S and P molecules is the same in processes $X$ and $Y$ (therefore, in order for $\brb{X(t),Y(t)}$ to abandon the chosen preorder relation, $X(t) = Y(t)$ should hold immediately before, and this is key to preventing the relation from being left). However, when $\text{S} \to \text{P}$ is replaced by a Michaelis--Menten reaction scheme, this is no longer the case, since now S and P molecules can also be converted into C molecules.

    \begin{figure}[htbp]
        \begin{center}
            \vskip 8 pt
            \begin{tikzpicture}
                \node[state] (S+E)  at (0,0)    {$\text{S}+\text{E}$}; 
                \node[state] (C)    at (2,0)    {$\text{C}$}; 
                \node[state] (E+P)  at (4,0)    {$\text{E}+\text{P}$};
                \node[state] (S)    at (1,-1)   {$\text{S}$};
                \node[state] (P)    at (3,-1)   {$\text{P}$};
                \path[-{Stealth[harpoon]}] 
                ([yshift = 2 px]S+E.east)  edge     node{} ([yshift = 2 px]C.west)
                ([yshift = -2 px]C.west)   edge     node{} ([yshift = -2 px]S+E.east);
                \path[-{Stealth}]
                (C)       edge                      node{} (E+P)
                (P)       edge                      node{} (S);
            \end{tikzpicture}
        \end{center}
    \caption{contrary to expectations, no non-trivial preordering structure can be found by our algorithm for this reaction network.}
    \label{fig:rmm}
    \end{figure}
\end{example}

\vskip 10 pt

\begin{example}[Signaling cascade]
\label{ex:sig}
    Signaling cascades are a class of reaction networks modeling the transmission of signals in cells. The topology of these networks can vary a lot depending on the specific model, but they all share the same basic structure: they consist of multiple similar layers, each triggering the subsequent one in a chain. They have been studied extensively in the reaction networks literature, see for example \cite{feliu2012algebraic}.

    In the particular instance of a signaling cascade that we analyzed with our methods, the number of layers is three. They are not connected by reactions, but by the species that they share. Each layer is a Michaelis--Menten mechanism transforming a substrate $\text{S}_i$ into a product $\text{P}_i$, which then acts as a catalyst in the conversion of $\text{S}_{i+1}$ into $\text{P}_{i+1}$ in the following layer. Our algorithm revealed seven preordering structures for this network, three of which are displayed in Figures~\ref{fig:sce}, \ref{fig:sc1}, and \ref{fig:sc2}.

    \begin{figure}[htbp]
        \begin{center}
            \vskip 8 pt
            \begin{tikzpicture}
                \node[state] (S1+P0)  at (0,0)    {$\textcolor{blue}{\text{S}_1}+\textcolor{blue}{\text{P}_0}$};
                \node[state] (C1)    at (2,0)    {$\textcolor{blue}{\text{C}_1}$}; 
                \node[state] (P0+P1)  at (4,0)    {$\textcolor{blue}{\text{P}_0}+\textcolor{blue}{\text{P}_1}$};
                \node[state] (S2+P1)  at (3,-1.5)    {$\textcolor{blue}{\text{S}_2}+\textcolor{blue}{\text{P}_1}$};
                \node[state] (C2)    at (5,-1.5)    {$\textcolor{blue}{\text{C}_2}$}; 
                \node[state] (P1+P2)  at (7,-1.5)    {$\textcolor{blue}{\text{P}_1}+\textcolor{Gray}{\text{P}_2}$};
                \node[state] (S3+P2)  at (6,-3)    {$\textcolor{Gray}{\text{S}_3}+\textcolor{Gray}{\text{P}_2}$};
                \node[state] (C3)    at (8,-3)    {$\textcolor{Gray}{\text{C}_3}$}; 
                \node[state] (P2+P3)  at (10,-3)    {$\textcolor{Gray}{\text{P}_2}+\textcolor{Gray}{\text{P}_3}$};
                \path[-{Stealth[harpoon]}]
                ([yshift = 2 px]S1+P0.east)  edge[color = blue]    node{} ([yshift = 2 px]C1.west)
                ([yshift = -2 px]C1.west)   edge[color = blue]    node{} ([yshift = -2 px]S1+P0.east)
                ([yshift = 2 px]S2+P1.east)  edge[color = blue]    node{} ([yshift = 2 px]C2.west)
                ([yshift = -2 px]C2.west)   edge[color = blue]    node{} ([yshift = -2 px]S2+P1.east)
                ([yshift = 2 px]S3+P2.east)  edge[color = Gray]    node{} ([yshift = 2 px]C3.west)
                ([yshift = -2 px]C3.west)   edge[color = Gray]    node{} ([yshift = -2 px]S3+P2.east);
                \path[-{Stealth}]
                (C1)       edge[color = blue]                    node{} (P0+P1)
                (C2)       edge[color = blue]                    node{} (P1+P2)
                (C3)       edge[color = Gray]                    node{} (P2+P3);
            \end{tikzpicture}
        \end{center}
    \caption{the rate constants satisfy the constraints $\kappa_{\textup{S}_1+\textup{P}_0 \to \textup{C}_1}^X = \kappa_{\textup{S}_1+\textup{P}_0 \to \textup{C}_1}^Y$, $\kappa_{\textup{C}_1 \to \textup{S}_1+\textup{P}_0}^X = \kappa_{\textup{C}_1 \to \textup{S}_1+\textup{P}_0}^Y$, $\kappa_{\textup{C}_1 \to \textup{P}_0+\textup{P}_1}^X = \kappa_{\textup{C}_1 \to \textup{P}_0+\textup{P}_1}^Y$, $\kappa_{\textup{S}_2+\textup{P}_1 \to \textup{C}_2}^X = \kappa_{\textup{S}_2+\textup{P}_1 \to \textup{C}_2}^Y$, $\kappa_{\textup{C}_2 \to \textup{S}_2+\textup{P}_1}^X = \kappa_{\textup{C}_2 \to \textup{S}_2+\textup{P}_1}^Y$ and $\kappa_{\textup{C}_2 \to \textup{P}_1+\textup{P}_2}^X = \kappa_{\textup{C}_2 \to \textup{P}_1+\textup{P}_2}^Y$, and by Theorem~\ref{thm:lsc} the species counts fulfill $X_{\textup{P}_0}(t) = Y_{\textup{P}_0}(t)$, $X_{\textup{S}_1}(t) = Y_{\textup{S}_1}(t)$, $X_{\textup{C}_1}(t) = Y_{\textup{C}_1}(t)$, $X_{\textup{P}_1}(t) = Y_{\textup{P}_1}(t)$, $X_{\textup{S}_2}(t) = Y_{\textup{S}_2}(t)$ and $X_{\textup{C}_2}(t) = Y_{\textup{C}_2}(t)$ for all $t \geq 0$ almost surely.}
    \label{fig:sce}
    \end{figure}

    \begin{figure}[htbp]
        \begin{center}
            \vskip 8 pt
            \begin{tikzpicture}
                \node[state] (S1+P0)  at (0,0)    {$\textcolor{blue}{\text{S}_1}+\textcolor{blue}{\text{P}_0}$};
                \node[state] (C1)    at (2,0)    {$\textcolor{blue}{\text{C}_1}$}; 
                \node[state] (P0+P1)  at (4,0)    {$\textcolor{blue}{\text{P}_0}+\textcolor{Green}{\text{P}_1}$};
                \node[state] (S2+P1)  at (3,-1.5)    {$\textcolor{red}{\text{S}_2}+\textcolor{Green}{\text{P}_1}$};
                \node[state] (C2)    at (5,-1.5)    {$\textcolor{red}{\text{C}_2}$}; 
                \node[state] (P1+P2)  at (7,-1.5)    {$\textcolor{Green}{\text{P}_1}+\textcolor{Gray}{\text{P}_2}$};
                \node[state] (S3+P2)  at (6,-3)    {$\textcolor{red}{\text{S}_3}+\textcolor{Gray}{\text{P}_2}$};
                \node[state] (C3)    at (8,-3)    {$\textcolor{Gray}{\text{C}_3}$}; 
                \node[state] (P2+P3)  at (10,-3)    {$\textcolor{Gray}{\text{P}_2}+\textcolor{Green}{\text{P}_3}$};
                \path[-{Stealth[harpoon]}]
                ([yshift = 2 px]S1+P0.east)  edge[color = blue]    node{} ([yshift = 2 px]C1.west)
                ([yshift = -2 px]C1.west)   edge[color = blue]    node{} ([yshift = -2 px]S1+P0.east)
                ([yshift = 2 px]S2+P1.east)  edge[color = blue]    node{} ([yshift = 2 px]C2.west)
                ([yshift = -2 px]C2.west)   edge[color = blue]    node{} ([yshift = -2 px]S2+P1.east)
                ([yshift = 2 px]S3+P2.east)  edge[color = Green]    node{} ([yshift = 2 px]C3.west)
                ([yshift = -2 px]C3.west)   edge[color = red]    node{} ([yshift = -2 px]S3+P2.east);
                \path[-{Stealth}]
                (C1)       edge[color = blue]                    node{} (P0+P1)
                (C2)       edge[color = Green]                    node{} (P1+P2)
                (C3)       edge[color = Green]                    node{} (P2+P3);
            \end{tikzpicture}
        \end{center}
    \caption{the rate constants satisfy the constraints $\kappa_{\textup{S}_1+\textup{P}_0 \to \textup{C}_1}^X = \kappa_{\textup{S}_1+\textup{P}_0 \to \textup{C}_1}^Y$, $\kappa_{\textup{C}_1 \to \textup{S}_1+\textup{P}_0}^X = \kappa_{\textup{C}_1 \to \textup{S}_1+\textup{P}_0}^Y$, $\kappa_{\textup{C}_1 \to \textup{P}_0+\textup{P}_1}^X = \kappa_{\textup{C}_1 \to \textup{P}_0+\textup{P}_1}^Y$, $\kappa_{\textup{S}_2+\textup{P}_1 \to \textup{C}_2}^X = \kappa_{\textup{S}_2+\textup{P}_1 \to \textup{C}_2}^Y$, $\kappa_{\textup{C}_2 \to \textup{S}_2+\textup{P}_1}^X = \kappa_{\textup{C}_2 \to \textup{S}_2+\textup{P}_1}^Y$, $\kappa_{\textup{C}_2 \to \textup{P}_1+\textup{P}_2}^X \leq \kappa_{\textup{C}_2 \to \textup{P}_1+\textup{P}_2}^Y$, $\kappa_{\textup{S}_3+\textup{P}_2 \to \textup{C}_3}^X \leq \kappa_{\textup{S}_3+\textup{P}_2 \to \textup{C}_3}^Y$, $\kappa_{\textup{C}_3 \to \textup{S}_3+\textup{P}_2}^X \geq \kappa_{\textup{C}_3 \to \textup{S}_3+\textup{P}_2}^Y$ and $\kappa_{\textup{C}_3 \to \textup{P}_2+\textup{P}_3}^X \leq \kappa_{\textup{C}_3 \to \textup{P}_2+\textup{P}_3}^Y$, and by Theorem~\ref{thm:lsc} the species counts fulfill $X_{\textup{P}_0}(t) = Y_{\textup{P}_0}(t)$, $X_{\textup{S}_1}(t) = Y_{\textup{S}_1}(t)$, $X_{\textup{C}_1}(t) = Y_{\textup{C}_1}(t)$, $X_{\textup{P}_1}(t) \leq Y_{\textup{P}_1}(t)$, $X_{\textup{S}_2}(t) \geq Y_{\textup{S}_2}(t)$, $X_{\textup{C}_2}(t) \geq Y_{\textup{C}_2}(t)$, $X_{\textup{S}_3}(t) \geq Y_{\textup{S}_3}(t)$ and $X_{\textup{P}_3}(t) \leq Y_{\textup{P}_3}(t)$ for all $t \geq 0$ almost surely.}
    \label{fig:sc1}
    \end{figure}

    \begin{figure}[htbp]
        \begin{center}
            \vskip 8 pt
            \begin{tikzpicture}
                \node[state] (S1+P0)  at (0,0)    {$\textcolor{blue}{\text{S}_1}+\textcolor{blue}{\text{P}_0}$};
                \node[state] (C1)    at (2,0)    {$\textcolor{blue}{\text{C}_1}$}; 
                \node[state] (P0+P1)  at (4,0)    {$\textcolor{blue}{\text{P}_0}+\textcolor{blue}{\text{P}_1}$};
                \node[state] (S2+P1)  at (3,-1.5)    {$\textcolor{blue}{\text{S}_2}+\textcolor{blue}{\text{P}_1}$};
                \node[state] (C2)    at (5,-1.5)    {$\textcolor{blue}{\text{C}_2}$}; 
                \node[state] (P1+P2)  at (7,-1.5)    {$\textcolor{blue}{\text{P}_1}+\textcolor{Green}{\text{P}_2}$};
                \node[state] (S3+P2)  at (6,-3)    {$\textcolor{red}{\text{S}_3}+\textcolor{Green}{\text{P}_2}$};
                \node[state] (C3)    at (8,-3)    {$\textcolor{red}{\text{C}_3}$}; 
                \node[state] (P2+P3)  at (10,-3)    {$\textcolor{Green}{\text{P}_2}+\textcolor{Green}{\text{P}_3}$};
                \path[-{Stealth[harpoon]}]
                ([yshift = 2 px]S1+P0.east)  edge[color = blue]    node{} ([yshift = 2 px]C1.west)
                ([yshift = -2 px]C1.west)   edge[color = blue]    node{} ([yshift = -2 px]S1+P0.east)
                ([yshift = 2 px]S2+P1.east)  edge[color = blue]    node{} ([yshift = 2 px]C2.west)
                ([yshift = -2 px]C2.west)   edge[color = blue]    node{} ([yshift = -2 px]S2+P1.east)
                ([yshift = 2 px]S3+P2.east)  edge[color = blue]    node{} ([yshift = 2 px]C3.west)
                ([yshift = -2 px]C3.west)   edge[color = blue]    node{} ([yshift = -2 px]S3+P2.east);
                \path[-{Stealth}]
                (C1)       edge[color = blue]                    node{} (P0+P1)
                (C2)       edge[color = blue]                    node{} (P1+P2)
                (C3)       edge[color = Green]                    node{} (P2+P3);
            \end{tikzpicture}
        \end{center}
    \caption{the rate constants satisfy the constraints $\kappa_{\textup{S}_1+\textup{P}_0 \to \textup{C}_1}^X = \kappa_{\textup{S}_1+\textup{P}_0 \to \textup{C}_1}^Y$, $\kappa_{\textup{C}_1 \to \textup{S}_1+\textup{P}_0}^X = \kappa_{\textup{C}_1 \to \textup{S}_1+\textup{P}_0}^Y$, $\kappa_{\textup{C}_1 \to \textup{P}_0+\textup{P}_1}^X = \kappa_{\textup{C}_1 \to \textup{P}_0+\textup{P}_1}^Y$, $\kappa_{\textup{S}_2+\textup{P}_1 \to \textup{C}_2}^X = \kappa_{\textup{S}_2+\textup{P}_1 \to \textup{C}_2}^Y$, $\kappa_{\textup{C}_2 \to \textup{S}_2+\textup{P}_1}^X = \kappa_{\textup{C}_2 \to \textup{S}_2+\textup{P}_1}^Y$, $\kappa_{\textup{C}_2 \to \textup{P}_1+\textup{P}_2}^X = \kappa_{\textup{C}_2 \to \textup{P}_1+\textup{P}_2}^Y$, $\kappa_{\textup{S}_3+\textup{P}_2 \to \textup{C}_3}^X = \kappa_{\textup{S}_3+\textup{P}_2 \to \textup{C}_3}^Y$, $\kappa_{\textup{C}_3 \to \textup{S}_3+\textup{P}_2}^X = \kappa_{\textup{C}_3 \to \textup{S}_3+\textup{P}_2}^Y$ and $\kappa_{\textup{C}_3 \to \textup{P}_2+\textup{P}_3}^X \leq \kappa_{\textup{C}_3 \to \textup{P}_2+\textup{P}_3}^Y$, and by Theorem~\ref{thm:lsc} the species counts fulfill $X_{\textup{P}_0}(t) = Y_{\textup{P}_0}(t)$, $X_{\textup{S}_1}(t) = Y_{\textup{S}_1}(t)$, $X_{\textup{C}_1}(t) = Y_{\textup{C}_1}(t)$, $X_{\textup{P}_1}(t) = Y_{\textup{P}_1}(t)$, $X_{\textup{S}_2}(t) = Y_{\textup{S}_2}(t)$, $X_{\textup{C}_2}(t) = Y_{\textup{C}_2}(t)$, $X_{\textup{P}_2}(t) \leq Y_{\textup{P}_2}(t)$, $X_{\textup{S}_3}(t) \geq Y_{\textup{S}_3}(t)$, $X_{\textup{C}_3}(t) \geq Y_{\textup{C}_3}(t)$ and $X_{\textup{P}_3}(t) \leq Y_{\textup{P}_3}(t)$ for all $t \geq 0$ almost surely.}
    \label{fig:sc2}
    \end{figure}

    The result presented in Figure~\ref{fig:sce} is peculiar relative to the others, as it does not involve any inequalities. A possible interpretation is that in this model, if certain rate constants are not changed, we do not expect the behavior of selected species to change, regardless of how the remaining parameters may be increased or decreased and of how the remaining species might be affected by these changes. What is then highlighted by this preordering structure -- that in this case could also be called \textit{equivalence structure}, since the underlying binary relation on the state space of the CTMCs is now an equivalence relation -- is how the dynamics of some species depends (or does not depend) on the different rate constants and on the dynamics of the other species. For instance, in the specific example of Figure~\ref{fig:sce}, we can say that the time evolution of the number of molecules of the species appearing in the third layer of the cascade does not influence whatsoever what happens to the remaining species.
\end{example}

\vskip 10 pt

\newpage\begin{example}[Population dynamics]
\label{ex:lkv}
    A precursory example of a reaction network is the Lotka--Volterra model, dating back to 1920 and describing the predation of one species by another (see \cite{murray2007mathematical} for an overview). Since then, many alternative population dynamics frameworks have been developed to describe different situations in the same spirit. Here a symmetric model is presented, similar to the one studied in \cite{gabel2013survival} but with the introduction of inflows and outflows (instead of quadratic decays), in which two populations A and B compete both with the same role. Individuals can \emph{enter} ($0 \to \text{A}$, $0 \to \text{B}$) or \emph{exit} the system ($\text{A} \to 0$, $\text{B} \to 0$), \emph{breed} ($\text{A} \to 2\text{A}$, $\text{B} \to 2\text{B}$), or \emph{die} due to competition with the other species ($\text{A}+\text{B} \to \text{B}$, $\text{A}+\text{B} \to \text{A}$). The symmetry of the model is reflected in the preordering structure displayed in Figure~\ref{fig:lkv}: both populations thrive -- at the expense of the other species -- when the corresponding influx and birth rates grow and outflow and competitive death rates decline.

    \begin{figure}[htbp]
        \begin{center}
            \vskip 8 pt
            \begin{tikzpicture}
                \node[state] (0)    at (0,-1.5)     {$0$}; 
                \node[state] (A)    at (2,0)        {$\textcolor{Green}{\text{A}}$}; 
                \node[state] (2A)   at (4,0)        {$2\textcolor{Green}{\text{A}}$};
                \node[state] (B)    at (2,-3)       {$\textcolor{red}{\text{B}}$}; 
                \node[state] (2B)   at (4,-3)       {$2\textcolor{red}{\text{B}}$};
                \node[state] (A+B)  at (2,-1.5)     {$\textcolor{Green}{\text{A}}+\textcolor{red}{\text{B}}$};
                \path[-{Stealth[harpoon]}]
                ([yshift = 7 px]0.east)     edge[color = Green]     node{} ([yshift = 1 px]A.west)
                ([yshift = -5 px]A.west)    edge[color = red]     node{} ([yshift = 1 px]0.east)
                ([yshift = -1 px]0.east)    edge[color = red]     node{} ([yshift = 5 px]B.west)
                ([yshift = -1 px]B.west)    edge[color = Green]     node{} ([yshift = -7 px]0.east);
                \path[-{Stealth}]
                (A)     edge[color = Green]     node{}  (2A)
                (B)     edge[color = red]       node{}  (2B)
                (A+B)   edge[color = Green]     node{}  (A)
                        edge[color = red]       node{}  (B);
            \end{tikzpicture}
        \end{center}
    \caption{the rate constants satisfy $\kappa_{0 \to \textup{A}}^X \leq \kappa_{0 \to \textup{A}}^Y$, $\kappa_{\textup{A} \to 0}^X \geq \kappa_{\textup{A} \to 0}^Y$, $\kappa_{\textup{A} \to 2\textup{A}}^X \leq \kappa_{\textup{A} \to 2\textup{A}}^Y$, $\kappa_{\textup{A}+\textup{B} \to \textup{B}}^X \geq \kappa_{\textup{A}+\textup{B} \to \textup{B}}^Y$, $\kappa_{0 \to \textup{B}}^X \geq \kappa_{0 \to \textup{B}}^Y$, $\kappa_{\textup{B} \to 0}^X \leq \kappa_{\textup{B} \to 0}^Y$, $\kappa_{\textup{B} \to 2\textup{B}}^X \geq \kappa_{\textup{B} \to 2\textup{B}}^Y$ and $\kappa_{\textup{A}+\textup{B} \to \textup{A}}^X \leq \kappa_{\textup{A}+\textup{B} \to \textup{A}}^Y$, and by Theorem~\ref{thm:lsc} the species counts fulfill $X_\textup{A}(t) \leq Y_\textup{A}(t)$ and $X_\textup{B}(t) \geq Y_\textup{B}(t)$ for all $t \geq 0$ almost surely.}
    \label{fig:lkv}
    \end{figure}

    As a general comment following Remark~\ref{rem:dif}, note that in all the results we present, the smaller terms of every rate-constant inequality can always be set to zero. This is essentially equivalent to removing the corresponding reaction \emph{only} from the network associated with one of $X$ and $Y$, and thus comparing two CTMCs based on \emph{different} reaction networks (and possibly with different parameters governing the common reactions). If one starts with two different reaction networks, the strategy to look for a potential comparison is to apply the algorithm to their union. It must then be checked whether the resulting rate-constant inequalities are compatible with the two sets of reaction removals required to recover the original networks. On the other hand, if a preordering structure for a single reaction network has already been found and one wants to study its implications for comparing different networks, the task is simpler: such heterogeneous comparisons can be deduced directly from the rate-constant inequalities by setting some of the smaller terms to zero. For example, in the preordering structure shown in Figure~\ref{fig:lkv} we can choose $\kappa_{\textup{A}+\textup{B} \to \textup{A}}^X = 0$ and $\kappa_{\textup{A}+\textup{B} \to \textup{B}}^Y = 0$. The remaining rate constants could take any set of values satisfying the inequalities of Figure~\ref{fig:lkv}, but for simplicity we can assume them equal in $X$ and $Y$. Under these assumptions, $X$ models the dynamics of a network obtained from the original by removing $\textup{A}+\textup{B} \to \textup{A}$, describing a scenario in which population B can induce mortality in population A but not vice versa. Conversely, $Y$ models the opposite scenario, with A now dominant and B unable to strike back. The resulting ordering of species counts presented by Figure~\ref{fig:lkv} (valid in the general case) aligns perfectly with this situation: as dominance shifts from B to A, the first population decreases while the second increases. This general idea will be further discussed and applied in Example~\ref{ex:erg}, where it will allow us to transfer desirable properties from one reaction network to another.
\end{example}

\vskip 10 pt

\begin{example}[Histone modification circuit]
\label{ex:his}
    Figure~\ref{fig:his} presents another relevant example drawn from \cite{campos2023comparison}, consisting of a SRN modeling gene regulation through histone modifications. The system includes three nucleosome states -- unmodified (D), repressed (R), and active (A) -- with a fixed total number of nucleosomes. Molecules of both R and A species can act as catalysts, either promoting the conversion of a D molecule into their own type or the reversion of a molecule of the other type back to D. An additional species P represents a gene product that enhances the conversion of D into A. Molecules of P are produced in the presence of active genes and degraded at a constant rate, creating a positive feedback loop that reinforces activation. See \cite{dodd2007theoretical} and \cite{bruno2022epigenetic} for more details on this model.
    
    \begin{figure}[htbp]
        \begin{center}
            \vskip 8 pt
            \begin{tikzpicture}
                \node[state] (R)    at (0,0)        {$\textcolor{red}{\text{R}}$}; 
                \node[state] (D)    at (2,0)        {$\textcolor{Gray}{\text{D}}$}; 
                \node[state] (A)    at (4,0)        {$\textcolor{Green}{\text{A}}$};
                \node[state] (0)    at (6,0)        {$0$};
                \node[state] (R+D)  at (1,-1)     {$\textcolor{red}{\text{R}}+\textcolor{Gray}{\text{D}}$};
                \node[state] (D+A)  at (3,-1)     {$\textcolor{Gray}{\text{D}}+\textcolor{Green}{\text{A}}$};
                \node[state] (2R)   at (0,-2.5)       {$2\textcolor{red}{\text{R}}$}; 
                \node[state] (A+R)  at (2,-2.5)        {$\textcolor{red}{\text{R}}+\textcolor{Green}{\text{A}}$}; 
                \node[state] (2A)   at (4,-2.5)       {$2\textcolor{Green}{\text{A}}$};
                \node[state] (P+D)  at (2,1.5)      {$\textcolor{Green}{\text{P}}+\textcolor{Gray}{\text{D}}$};
                \node[state] (A+P)  at (4,1.5)      {$\textcolor{Green}{\text{A}}+\textcolor{Green}{\text{P}}$};
                \node[state] (P)    at (6,1.5)        {$\textcolor{Green}{\text{P}}$};
                \path[-{Stealth[harpoon]}]
                ([yshift = 2 px]R.east)     edge[color = Green]    node{} ([yshift = 2 px]D.west)
                ([yshift = -2 px]D.west)    edge[color = red]    node{} ([yshift = -2 px]R.east)
                ([yshift = 2 px]D.east)     edge[color = Green]    node{} ([yshift = 2 px]A.west)
                ([yshift = -2 px]A.west)    edge[color = red]    node{} ([yshift = -2 px]D.east);
                \path[-{Stealth}]
                (A)     edge[color = Green]     node{}  (A+P)
                (P+D)   edge[color = Green]     node{}  (A+P)
                (P)     edge[color = red]       node{}  (0)
                (A+R)   edge[color = red]       node{}  (R+D)
                (R+D)   edge[color = red]       node{}  (2R)
                (A+R)   edge[color = Green]     node{}  (D+A)
                (D+A)   edge[color = Green]     node{}  (2A);
            \end{tikzpicture}
        \end{center}
    \caption{the rate constants satisfy $\kappa_{\textup{R} \to \textup{D}}^X \leq \kappa_{\textup{R} \to \textup{D}}^Y$, $\kappa_{\textup{D} \to \textup{R}}^X \geq \kappa_{\textup{D} \to \textup{R}}^Y$, $\kappa_{\textup{D} \to \textup{A}}^X \leq \kappa_{\textup{D} \to \textup{A}}^Y$, $\kappa_{\textup{A} \to \textup{D}}^X \geq \kappa_{\textup{A} \to \textup{D}}^Y$, $\kappa_{\textup{A} \to \textup{A}+\textup{P}}^X \leq \kappa_{\textup{A} \to \textup{A}+\textup{P}}^Y$, $\kappa_{\textup{P} \to 0}^X \geq \kappa_{\textup{P} \to 0}^Y$, $\kappa_{\textup{P}+\textup{D} \to \textup{A}+\textup{P}}^X \leq \kappa_{\textup{P}+\textup{D} \to \textup{A}+\textup{P}}^Y$, $\kappa_{\textup{R}+\textup{A} \to \textup{R}+\textup{D}}^X \geq \kappa_{\textup{R}+\textup{A} \to \textup{R}+\textup{D}}^Y$, $\kappa_{\textup{R}+\textup{D} \to 2\textup{R}}^X \geq \kappa_{\textup{R}+\textup{D} \to 2\textup{R}}^Y$, $\kappa_{\textup{R}+\textup{A} \to \textup{D}+\textup{A}}^X \leq \kappa_{\textup{R}+\textup{A} \to \textup{D}+\textup{A}}^Y$ and $\kappa_{\textup{D}+\textup{A} \to 2\textup{A}}^X \leq \kappa_{\textup{D}+\textup{A} \to 2\textup{A}}^Y$, and by Theorem~\ref{thm:lsc} the species counts fulfill $X_\textup{R}(t) \geq Y_\textup{R}(t)$, $X_\textup{A}(t) \leq Y_\textup{A}(t)$ and $X_\textup{P}(t) \leq Y_\textup{P}(t)$ for all $t \geq 0$ almost surely.}
    \label{fig:his}
    \end{figure}
    
    As one can deduce from Figure~\ref{fig:his}, our algorithm obtained the same preorder on the species counts as in \cite{campos2023comparison}, but with weaker conditions on the MAKs. In \cite[Example~4.5]{campos2023comparison}, only the rate constant associated with $\text{A} \to \text{A}+\text{P}$ is allowed to increase when going from $X$ to $Y$, while all the others must remain equal. By contrast, in our result, all pairs of homologous rate constants are only compared through an inequality. It is also interesting to observe that we are able to control the copy numbers of species R and A even in the presence of reactions such as $\text{R} \rightleftharpoons \text{D} \rightleftharpoons \text{A}$, where R and A interact closely with D -- a species on which we have no general knowledge. This is possible thanks to the extra information that we acquire when the coupled process $\rb{X,Y}$ touches one of the boundaries of the preorder relation, encoded in the first $m$ rows of the matrices $A$ and $B$ computed in Theorem~\ref{thm:lsc}. In other words, we do not know in general whether the abundance of D is higher in $X$ or $Y$, but this is not a concern because we know it in all situations in which the occurrence of one reaction could lead $\rb{X,Y}$ outside the preorder relation. In those cases, $X_\text{D}(t)$ and $Y_\text{D}(t)$ are ordered exactly as we need them to be to preserve the preorder on the copy numbers of R and A. A similar situation can also be observed in the first result presented for Michaelis--Menten (see species E and C in Figure~\ref{fig:mm1} of Example~\ref{ex:mm2}) and in one of the preordering structures for the signaling cascade of Example~\ref{ex:sig} (species $\text{P}_2$ and $\text{C}_3$ of Figure~\ref{fig:sc1}).
\end{example}

\vskip 10 pt

\newpage\begin{example}[Ergodicity detection]
\label{ex:erg}
    In this last example, we will show again how our theory can also be used to compare CTMCs associated with different reaction networks, as already seen in Example~\ref{ex:lkv}. We will be able to easily prove the ergodicity of a SRN to which the deficiency theory cannot be directly applied, by means of a comparison with another SRN that fits into that framework. The reaction networks under study are depicted in Figure~\ref{fig:er1}.

    \begin{figure}[htbp]
        \begin{center}
            \vskip 8 pt
            \begin{tikzpicture}
            \node[state] (rn)    at (-0.8,1)  {$\rn:$}; 
            \node[state] (0)     at (0,0)     {$0$}; 
            \node[state] (A)     at (2,0)     {$\text{A}$}; 
            \node[state] (B)     at (0,2)     {$\text{B}$};
            \node[state] (rn')    at (4.2,1)   {$\bar{\rn}:$}; 
            \node[state] (0')     at (5,0)     {$0$}; 
            \node[state] (A')     at (7,0)     {$\text{A}$}; 
            \node[state] (B')     at (5,2)     {$\text{B}$};
            \node[state] (A+B')   at (7,2)     {$\text{A}+\text{B}$};
            \path[-{Stealth}]
            (0)     edge   node{}  (A)
            (A)     edge   node{}  (B)
            (B)     edge   node{}  (0)
            (0')     edge   node{}  (A')
            (A')     edge   node{}  (B')
            (B')     edge   node{}  (0')
            (A+B')   edge   node{}  (A');
        \end{tikzpicture}
        \end{center}
    \caption{the reaction network $\rn$ (left) is weakly reversible and has a deficiency of zero; its augmentation $\bar{\rn}$ (right) has neither property.}
    \label{fig:er1}
    \end{figure}

    Throughout this example, assume that all rate constants of the base network $\rn$ are positive. $\rn$ is a simple two-dimensional network, it is \textit{weakly reversible} and has a \textit{deficiency} of zero (the definitions of these two notions can be found in \cite{anderson2015stochastic}). It is well-known that MAK SRNs satisfying said properties are ergodic on each irreducible class for all choices of positive rate constants -- see \cite{anderson2010product} and \cite{anderson2018non}. One could then think that, by adding to the system any new reaction with non-positive reaction vector, ergodicity should be preserved. However, this is not always the case. There are multiple examples disproving this intuition, usually exploiting some boundary effect (see for example \cite{anderson2020tier} and \cite{agazzi2020seemingly} on this phenomenon). It is then non-trivial to say that $\bar{\rn}$, obtained from $\rn$ by adding the reaction $\text{A}+\text{B} \to \text{A}$, still originates ergodic MAK SRNs for all choices of rate constants. This can be proved by applying Theorem~\ref{thm:lsc} to $\bar{\rn}$, with
    \begin{equation*}
        M =
        \begin{pmatrix} 
            1 & 0 \\ 
            -1 & 0 \\ 
            0 & -1 
        \end{pmatrix},
    \end{equation*}
    a generic $K^Y$, and $K^X$ equal to $K^Y$ in all components except for $\kappa^X_{\text{A}+\text{B} \to \text{A}} = 0$. It follows from this equality that the process $X$ is actually based on $\rn$, because the rates of all transitions associated with $\text{A}+\text{B} \to \text{A}$ are zero, meaning that the reaction is not really present in the model. $X$ is then ergodic, for the reasons explained above. $Y$, on the other hand, is a general MAK SRN associated with $\bar{\rn}$, and thanks to Theorem~\ref{thm:lsc} we conclude that it can be coupled with $X$ in such a way that
    \begin{equation}
    \label{eq:erg}
        \P\brb{X_\text{A}(t) = Y_\text{A}(t), \ X_\text{B}(t) \geq Y_\text{B}(t) \ \ \forall t \geq 0} = 1.
    \end{equation}
    
    Note that the situation just described is a special case of the result displayed in Figure~\ref{fig:er2}, where two of the assumptions used so far, namely $\kappa_{\textup{B} \to 0}^X = \kappa_{\textup{B} \to 0}^Y$ and $\kappa_{\textup{A}+\textup{B} \to \textup{A}}^X =0$, are replaced by the less restrictive $\kappa_{\textup{B} \to 0}^X \leq \kappa_{\textup{B} \to 0}^Y$ and $\kappa_{\textup{A}+\textup{B} \to \textup{A}}^X \leq \kappa_{\textup{A}+\textup{B} \to \textup{A}}^Y$. This preordering structure for $\bar{\rn}$ was identified by our algorithm, together with two others, each describing the consequences of different sets of changes in the rate constants.

    \begin{figure}[htbp]
        \begin{center}
            \vskip 8 pt
            \begin{tikzpicture}
                \node[state] (0)    at (0,0)        {$0$}; 
                \node[state] (A)    at (2,0)        {$\textcolor{blue}{\text{A}}$}; 
                \node[state] (B)    at (0,2)        {$\textcolor{red}{\text{B}}$};
                \node[state] (A+B)    at (2,2)        {$\textcolor{blue}{\text{A}}+\textcolor{red}{\text{B}}$};
                \path[-{Stealth}]
                (0)     edge[color = blue]    node{}  (A)
                (A)     edge[color = blue]    node{}  (B)
                (B)     edge[color = Green]    node{}  (0)
                (A+B)   edge[color = Green]    node{}  (A);
            \end{tikzpicture}
        \end{center}
        \caption{the rate constants satisfy $\kappa_{0 \to \textup{A}}^X = \kappa_{0 \to \textup{A}}^Y$, $\kappa_{\textup{A} \to \textup{B}}^X = \kappa_{\textup{A} \to \textup{B}}^Y$, $\kappa_{\textup{B} \to 0}^X \leq \kappa_{\textup{B} \to 0}^Y$ and $\kappa_{\textup{A}+\textup{B} \to \textup{A}}^X \leq \kappa_{\textup{A}+\textup{B} \to \textup{A}}^Y$, and by Theorem~\ref{thm:lsc} the species counts fulfill $X_\textup{A}(t) = Y_\textup{A}(t)$ and $X_\textup{B}(t) \geq Y_\textup{B}(t)$ for all $t \geq 0$ almost surely.}
        \label{fig:er2}
    \end{figure}
    
    After applying Theorem~\ref{thm:lsc}, the missing link between \eqref{eq:erg} and the ergodicity of $Y$ is the tightness of the family of its time-$t$ distributions. The tightness of $Y$, however, can easily be derived from that of $X$ -- implied by its ergodicity -- via \eqref{eq:erg}. We then conclude that $Y$ is ergodic as well, by means of Prohorov's and Krylov--Bogoliubov theorems (that can be found in \cite{billingsley2013convergence} and \cite{da1996ergodicity}, respectively). Recall that $Y$ is any MAK SRN associated with $\bar{\rn}$, so we proved that adding the reaction $\text{A}+\text{B} \to \text{A}$ to $\rn$ does not break the stability of the system, even if it prevents us from applying the deficiency theory. The analysis of problems similar to this is typical in SRN theory, further motivating the study and application of the tools presented in this paper.
\end{example}

\section{Conclusions}

In this work, we established general comparison results for stochastic reaction networks with mass-action or, more generally, monotone kinetics and fixed functions $f_r$, under different choices of rate constants. We translated these results into an explicit and fully implementable algorithm. More precisely, in Theorem~\ref{thm:pcc} we identified a set of verifiable conditions under which two continuous-time Markov chains can be coupled so as to preserve a given relation for all times (up to the first explosion), and in Theorem~\ref{thm:lsc} we restricted our scope to stochastic reaction networks and matrix preorder relations. We then applied the algorithm to several examples, recovering known preordering structures, extending them, and introducing new ones. By doing so, we gained new insights into how the dynamics of the various models depend on the rates of the corresponding reactions. We also described (see Examples~\ref{ex:lkv} and \ref{ex:erg}) how our methods could be used to study the effects of adding or removing reactions instead of just changing the rates of the existing ones.

A key feature of our approach is the linearity -- hence low computational cost -- of the hypotheses of Theorem~\ref{thm:lsc}; checking them reduces to a series of feasibility problems in linear programming and few matrix multiplications. This, together with a high degree of parallelizability, allows the search for admissible preordering structures to be efficient, making it possible to explore a large number of candidates as illustrated by the computations in Section~\ref{sub:imp}. The algorithm therefore provides an effective computational tool for systematically identifying a set (possibly not complete) of preordering structures that are compatible with a given network.

At the same time, our analysis highlights limitations of our sufficient criterion. For some reaction networks (see Examples~\ref{ex:sis} and \ref{ex:rmm}), our algorithm finds no non-trivial preordering structure within the searched family. We believe that the major limitation is given by considering a restrictive almost-sure pathwise comparison rather than the specific choice of sufficient conditions we made. In this sense, a weaker but more flexible ordering principle would be preferable and, for this reason, an important direction for future work is to investigate alternative and more relaxed comparison notions. One possibility could be to simply compare the expected values of the two processes over time. Developing such weaker yet more robust comparison principles may broaden considerably the range of reaction networks for which informative monotonicity properties can be obtained.

\begin{appendices}
    \allowdisplaybreaks
    \section{Marginalizability of CTMCs}
    \label{sec:app}
    
    The problem of marginalizability of a CTMC is a special case of lumpability. Roughly speaking, the general question is: given a CTMC, is the Markov property preserved under a lumping of its state space? In the context of marginalizability, the state space is multidimensional, and states are grouped according to the value of a subset of their components. The notion of lumpability was first introduced for discrete-time Markov chains by \cite{kemeny1969finite} and later adapted to the continuous-time setting and studied in multiple works, with assumptions of varying strength on the state space (finite or countable) and on the rate matrix (uniformizable or non-explosive). To the best of our knowledge, the most general treatment of the subject is that of \cite{ball1993lumpability}, which still excludes explosive processes. An extension of \cite[Theorem~3.1]{ball1993lumpability} to potentially explosive chains is therefore given in this appendix; this result is essential to prove our Theorem~\ref{thm:pcc} in its full generality.

    \begin{samepage}
    \begin{theorem}
    \label{thm:app}
        Let $E \subseteq \Z^d$. Consider a CTMC $W = \brb{W(t)}_{t \geq 0}$ with state space $E \times E$ and rate matrix $Q^W$ having finite diagonal elements (i.e.\ no instantaneous states), such that for all $x, y, x', y' \in E$ with $x \neq x'$ and $y \neq y'$
        \begin{equation}
        \label{eq:indy}
            \sum_{z' \in E} Q^W\brb{(x,y), (x',z')} = q^X(x,x')
        \end{equation}
        is independent of $y$, and
        \begin{equation}
        \label{eq:indx}
            \sum_{z' \in E} Q^W\brb{(x,y), (z',y')} = q^Y(y,y')
        \end{equation}
        is independent of $x$. Then, there exist two CTMCs $X = \brb{X(t)}_{t \geq 0}$ and $Y = \brb{Y(t)}_{t \geq 0}$ with state space $E$ and transition rates $q^X$ and $q^Y$ respectively, such that $\tau^X \wedge \tau^Y = \tau^W$ and
        \begin{equation*}
            W(t) = \brb{X(t), Y(t)} \ \ \forall t \in [0,\tau^W).
        \end{equation*}
    \end{theorem}
    \end{samepage}
    
    \begin{proof}
        First, note that all three mentioned CTMCs may explode in finite time, but the statement of the theorem only concerns their behavior up to the first explosion. For the purpose of this proof and for simplicity, we can then consider an extended state space $E_\Delta = E \cup \cb{\Delta}$ for $X$ and $Y$ and construct the chains in their minimal versions, in which they are absorbed by the cemetery state $\Delta$ at explosion. This ensures a rigorous and convenient framework for our argument, without affecting the chains' behavior before the explosion.
    
        Let us extend $Q^W$ to a new rate matrix $Q$ on the state space $E_\Delta \times E_\Delta$ in the following way, for all $x,y,x',y' \in E$:
        \begin{align}
        \label{eq:defq}
            Q\brb{(x,y), (x',y')} & = Q^W\brb{(x,y), (x',y')}, \\
            Q\brb{(x,\Delta), (x',\Delta)} & = q^X(x,x'), \notag \\
            Q\brb{(\Delta,y), (\Delta,y')} & = q^Y(y,y'), \notag
        \end{align}
        where $q^X(x,x')$ and $q^Y(y,y')$ are defined in terms of $Q^W$ in \eqref{eq:indy} and \eqref{eq:indx} when $x \neq x'$ and $y \neq y'$, while $q^X(x,x) = - \sum\limits_{x' \neq x} q^X(x,x')$ and $q^Y(y,y) = - \sum\limits_{y' \neq y} q^Y(y,y')$ follow the usual convention. Note that no transitions are defined among the four sets $E \times E$, $E \times \cb{\Delta}$, $\cb{\Delta} \times E$, and $\cb{\Delta} \times \cb{\Delta}$, partitioning $E_\Delta \times E_\Delta$: it is only possible to move from one to another via explosion.
        
        Let $\rb{X,Y} = \Brb{\brb{X(t), Y(t)}}_{t \geq 0}$ be a CTMC with transition rates $Q$, and let $\tau$ be its explosion time (a shorthand for $\tau^{(X,Y)}$) and $\tau^X$, $\tau^Y$ the ones of the marginal processes $X$ and $Y$. We observe that the identity $\tau = \tau^X \wedge \tau^Y$ follows immediately from the definition of the explosion time given in \eqref{eq:exp}. In order to define the processes after explosion, set
        \begin{equation}
        \label{eq:cem}
            X(\tau^X) = \Delta \ \ \text{and} \ \ Y(\tau^Y) = \Delta.
        \end{equation}
        Now suppose that $\brb{X(0), Y(0)} \in E \times E$. In the event that $\tau = \tau^X = \tau^Y$, the behavior of the chain after explosion is already fully determined. On the other hand, if $\tau = \tau^Y < \tau^X$, \eqref{eq:cem} only gives us $Y$ at time $\tau$, so we need to assign a value to $X$ as well. In that case, define $X(\tau) = \limsup\limits_{t \to \tau^-} X(t)$ (component-wise). Analogously, set $Y(\tau) = \limsup\limits_{t \to \tau^-} Y(t)$ if $\tau = \tau^X < \tau^Y$.
        
        Since $Q$ coincides with $Q^W$ on $E \times E$, we can couple $(X,Y)$ and $W$ so that $\tau^X \wedge \tau^Y = \tau = \tau^W$ and 
        \begin{equation*}
            W(t) = \brb{X(t), Y(t)} \ \ \forall t \in [0,\tau^W).
        \end{equation*}
        What remains to be proved is that $X$ and $Y$ themselves are CTMCs, with rates $q^X$ and $q^Y$. Let us show it for $X$, the argument for $Y$ being completely analogous.
        
        Choose a norm $\nrm{\cdot}$ on $\R^d$. For any positive integer $n$, let
        \begin{align*}
            \tau^X_n &= \inf \bcb{t \geq 0 \st \nrm{X(t)} > n}, \\
            \tau^Y_n &= \inf \bcb{t \geq 0 \st \nrm{Y(t)} > n}, \\
            \tau_n &= \tau^X_n \wedge \tau^Y_n.
        \end{align*}
        Define a truncated CTMC $\rb{X_n, Y_n} = \Brb{\brb{X_n(t), Y_n(t)}}_{t \geq 0}$ on the state space $E \times E$ via the transition rates
        \begin{equation}
        \label{eq:defqn}
            Q_n\brb{(x,y), (x',y')} = 
            \begin{cases}
                Q\brb{(x,y), (x',y')} & \text{if} \ \nrm{x} \leq n, \, \nrm{y} \leq n, \\
                q^X(x,x') & \text{if} \ \nrm{x} \leq n, \, \nrm{y} > n, \, y = y', \\
                q^Y(y,y') & \text{if} \ \nrm{x} > n, \, \nrm{y} \leq n, \, x = x', \\
                0 & \text{otherwise},
            \end{cases}
        \end{equation}
        for all $x,y,x',y' \in E$. This definition makes it possible to couple $\rb{X, Y}$ with $\rb{X_n, Y_n}$ in such a way that
        \begin{equation}
        \label{eq:cpl1}
            \brb{X(t), Y(t)} = \brb{X_n(t), Y_n(t)} \ \ \forall t \in [0, \tau_n].
        \end{equation}
        In particular, $X(t \wedge \tau_n) = X_n(t\wedge \tau_n)$. Since for all $x,y,x' \in E$ with $\nrm{x} \leq n$ and $\nrm{y} > n$ we have the identities
        \begin{equation*}
            Q\brb{(x,\Delta), (x',\Delta)} = q^X(x,x') = Q_n\brb{(x,y), (x',y)},
        \end{equation*}
        we can further couple $\rb{X, Y}$ and $\rb{X_n, Y_n}$ so that, if $X(\tau^Y)=X_n(\tau^Y_n)$ and $\tau^Y < \tau^X_n$, then 
        \begin{equation}
        \label{eq:cpl2}
            X\brb{(\tau^Y+h) \wedge \tau^X_n} = X_n(\tau^Y_n+h) \ \ \forall h \geq 0.
        \end{equation}
        
        Now, by combining \eqref{eq:indy}, \eqref{eq:defq}, and \eqref{eq:defqn} we derive
        \begin{equation}
        \label{eq:indyn}
            \sum_{z' \in E} Q_n\brb{(x,y), (x',z')} = q^X(x,x') \mathbbm{1}_{\mathrm{B}_n}(x),
        \end{equation}
        for all $x,y,x' \in E$ with $x \neq x'$, where $\mathrm{B}_n = \cb{z \in \Z^d \st \nrm{z} \leq n}$ is the discrete ball of radius $n$ centered at the origin. As a consequence, by classical results on lumpability of Markov chains such as \cite[Theorem~3.1]{ball1993lumpability}, $X_n$ is a CTMC with rates $q^X(x,x') \mathbbm{1}_{\mathrm{B}_n}(x)$. More precisely, what is stated in \cite{ball1993lumpability} is the equivalence between \eqref{eq:indyn} and $X_n$ being a CTMC, under the additional hypotheses of irreducibility and positive recurrence of $\rb{X_n, Y_n}$. However, these assumptions are used only to prove the necessity of \eqref{eq:indyn} but not its sufficiency, which is what we need here.
        
        To obtain our desired result about $X$, denote by $\hat{X}$ a minimal CTMC on the state space $E_\Delta$, with transition rates $q^X$ and the same initial distribution as $X$. Let $\tau^{\hat{X}}$ be its explosion time, and define $\tau^{\hat{X}}_n = \inf \cb{t \geq 0 \st \nrm{\hat{X}(t)} > n}$. To conclude the proof, it suffices to show, as explained in \cite[Chapter 2]{norris1998markov}, that for any finite collection of times $0 < t_1 < \dots < t_\ell$ and states $x_1, \dots, x_\ell \in E_\Delta$ we have
        \begin{equation}
        \label{eq:dist}
            \P\brb{X(t_1)= x_1, \dots, X(t_\ell) = x_\ell} = \P\brb{\hat{X}(t_1) = x_1, \dots, \hat{X}(t_\ell) = x_\ell},
        \end{equation}
        i.e.\ $X$ and $\hat{X}$ have identical finite-dimensional distributions.
        
        As a consequence of \eqref{eq:cpl1}, for $0 \leq j \leq \ell$ we have
        \begin{align}
        \label{eq:case1}
             \hspace{70 pt} & \hspace{-70 pt} \P \brb{X(t_1) = x_1, \dots, X(t_\ell) = x_\ell, \ t_j < \tau^X \leq t_{j+1}, \ t_j < \tau^Y} \notag \\
             = & \ \P \Brb{\bigcap_{k=1}^j X(t_k) = x_k, \ t_j < \tau^X \leq t_{j+1}, \ t_j < \tau^Y} \notag \\
             & \cdot \P \Brb{\bigcap_{k=j+1}^\ell X(t_k) = x_k \Bigm| \bigcap_{k=1}^j X(t_k) = x_k, \ t_j < \tau^X \leq t_{j+1}, \ t_j < \tau^Y} \notag \\
             = & \lim_{n \to \infty} \P \Brb{\bigcap_{k=1}^j X_n(t_k)= x_k, \ t_j < \tau^X_n \leq t_{j+1}, \ t_j < \tau^Y_n} \prod_{k=j+1}^\ell \mathbbm{1}_{\cb{\Delta}}(x_k),
        \end{align}
        where we set $t_0 = 0$ and $t_{\ell + 1} = \infty$. To handle the extremal cases in the formula above and those that follow, note that, as usual, we consider empty sums, products, or intersections of events to be equal to 0, 1, or the whole sample space, respectively.
        
        Now, in the event that $\tau = \tau^Y < \tau^X$, we have that the total number of jumps of $X$ up to time $\tau$ is almost surely finite. Indeed, the definition of $\tau^X$ given in \eqref{eq:exp} ensures the existence of a (random) finite set $\Gamma$ such that $X(t) \in \Gamma$ for all $t \in [0,\tau)$. Since each $X_n$ coincides with $X$ on the interval $[0, \tau_n] \subset [0, \tau)$ and is a CTMC with rates $q^X(x,x') \mathbbm{1}_{\mathrm{B}_n}(x)$, there is a uniform bound on the number of jumps of $X_n = X$ in $[0, \tau_n]$, given for example by a Poisson process with rate $\max\limits_{x \in \Gamma} - q^X(x,x)$ evaluated at time $\tau$. Hence, the number of jumps of $X$ in all of $[0, \tau)$ is almost surely finite. This means that $X(\tau) = \limsup\limits_{t \to \tau^-} X(t)$ coincides with $X(t)$ for $t$ close enough to $\tau$.
        
        As a consequence, there exists an almost surely finite $N$ such that, for any $n > N$,
        \begin{equation*}
            \mathbbm{1}_{\cb{\tau^Y < \tau^X}} X(\tau^Y) = \mathbbm{1}_{\cb{\tau^Y < \tau^X}} X(\tau^Y_n) = \mathbbm{1}_{\cb{\tau^Y < \tau^X}} X_n(\tau^Y_n).
        \end{equation*}
        Thus, due to the coupling \eqref{eq:cpl2}, for  $n > N$ and $h \geq 0$,
        \begin{equation}
        \label{eq:cpl2b}
             \mathbbm{1}_{\cb{\tau^Y < \tau^X}} X\brb{(\tau^Y+h) \wedge \tau^X_n} = \mathbbm{1}_{\cb{\tau^Y < \tau^X}} X_n(\tau^Y_n+h).
        \end{equation}
        
        Using the last identity, for $0 \leq i < j \leq \ell$ we obtain
        \begin{align*}
            \hspace{15 pt} & \hspace{-15 pt} \P \brb{X(t_1)= x_1, \dots, X(t_\ell) = x_\ell, \ t_i < \tau^Y \leq t_{i+1}, \ t_j < \tau^X \leq t_{j+1}} \\
            = & \ \P \Brb{\bigcap_{k=1}^i X(t_k) = x_k, \ t_i < \tau^Y \leq t_{i+1}, \ t_j < \tau^X \leq t_{j+1}} \\
            & \cdot \P \Brb{\bigcap_{k=i+1}^j X(t_k) = x_k \Bigm| \bigcap_{k=1}^i X(t_k) = x_k, \ t_i < \tau^Y \leq t_{i+1}, \ t_j < \tau^X \leq t_{j+1}} \\
            & \cdot \P \Brb{\bigcap_{k=j+1}^\ell X(t_k) = x_k \Bigm| \bigcap_{k=1}^j X(t_k) = x_k, \ t_i < \tau^Y \leq t_{i+1}, \ t_j < \tau^X \leq t_{j+1}} \\
            = & \lim_{n \to \infty} \bbsb{\P \Brb{\bigcap_{k=1}^i X_n(t_k) = x_k, \ t_i < \tau^Y_n \leq t_{i+1}, \ t_j < \tau^X_n \leq t_{j+1}} \\
            & \cdot \P \Brb{\bigcap_{k=i+1}^j X_n\brb{(\tau^Y_n+t_k-\tau^Y)^+} = x_k \Bigm| \bigcap_{k=1}^i X_n(t_k) = x_k, \ t_i < \tau^Y_n \leq t_{i+1}, \ t_j < \tau^X_n \leq t_{j+1}}} \\
            & \cdot \prod_{k=j+1}^\ell \mathbbm{1}_{\cb{\Delta}}(x_k) \\
            = & \lim_{n \to \infty} \bbsb{\P \Brb{\bigcap_{k=1}^i X_n(t_k) = x_k, \ t_i < \tau^Y_n \leq t_{i+1}, \ t_j < \tau^X_n \leq t_{j+1}} \\
            & \cdot \P \Brb{\bigcap_{k=i+1}^j X_n(t_k) = x_k \Bigm| \bigcap_{k=1}^i X_n(t_k) = x_k, \ t_i < \tau^Y_n \leq t_{i+1}, \ t_j < \tau^X_n \leq t_{j+1}}} \\
            & \cdot \prod_{k=j+1}^\ell \mathbbm{1}_{\cb{\Delta}}(x_k) \\
            = & \lim_{n \to \infty} \P \Brb{\bigcap_{k=1}^j X_n(t_k) = x_k, \ t_i < \tau^Y_n \leq t_{i+1}, \ t_j < \tau^X_n \leq t_{j+1}} \prod_{k=j+1}^\ell \mathbbm{1}_{\cb{\Delta}}(x_k),
        \end{align*}
        where we set $t_0 = 0$ and $t_{\ell+1} = \infty$, as before. Combining this with \eqref{eq:case1}, for $0 \leq j \leq \ell$ we get
        \begin{align}
        \label{eq:last}
            \hspace{70 pt} & \hspace{-70 pt} \P \brb{X(t_1) = x_1, \dots, X(t_\ell) = x_\ell, \ t_j < \tau^X \leq t_{j+1}} \notag \\
            = & \sum_{i=0}^\ell \P \brb{X(t_1) = x_1, \dots, X(t_\ell) = x_\ell, \ t_j < \tau^X \leq t_{j+1}, \ t_i < \tau^Y \leq t_{i+1}} \notag \\
            = & \sum_{i=0}^{j-1} \P \brb{X(t_1) = x_1, \dots, X(t_\ell) = x_\ell, \ t_i < \tau^Y \leq t_{i+1}, \ t_j < \tau^X \leq t_{j+1}} \notag \\
            & + \P \brb{X(t_1) = x_1, \dots, X(t_\ell) = x_\ell, \ t_j < \tau^X \leq t_{j+1}, \ t_j < \tau^Y} \notag \\
            = & \lim_{n \to \infty} \bbsb{\sum_{i=0}^{j-1} \P \Brb{\bigcap_{k=1}^j X_n(t_k) = x_k, \ t_i < \tau^Y_n \leq t_{i+1}, \ t_j < \tau^X_n \leq t_{j+1}} \notag \\*
            & + \P \Brb{\bigcap_{k=1}^j X_n(t_k)= x_k, \ t_j < \tau^X_n \leq t_{j+1}, \ t_j < \tau^Y_n}} \prod_{k=j+1}^\ell \mathbbm{1}_{\cb{\Delta}}(x_k) \notag \\
            = & \lim_{n \to \infty} \P \Brb{\bigcap_{k=1}^j X_n(t_k) = x_k, \ t_j < \tau^X_n \leq t_{j+1}} \prod_{k=j+1}^\ell \mathbbm{1}_{\cb{\Delta}}(x_k) \notag \\
            = & \lim_{n \to \infty} \P \Brb{\bigcap_{k=1}^j \hat{X}(t_k) = x_k, \ t_j < \tau^{\hat{X}}_n \leq t_{j+1}} \prod_{k=j+1}^\ell \mathbbm{1}_{\cb{\Delta}}(x_k) \notag \\
            = & \ \P \Brb{\bigcap_{k=1}^j \hat{X}(t_k) = x_k, \ t_j < \tau^{\hat{X}} \leq t_{j+1}} \prod_{k=j+1}^\ell \mathbbm{1}_{\cb{\Delta}}(x_k) \notag \\
            = & \ \P \brb{\hat{X}(t_1) = x_1, \dots, \hat{X}(t_\ell)= x_\ell, \ t_j < \tau^{\hat{X}} \leq t_{j+1}},
        \end{align}
        where in the third-last equality we used that, for every fixed $T>0$,
        \begin{equation*}
            \rb{\tau_n^X \wedge T} - \rb{\tau_n^{X_n} \wedge T} \xrightarrow[n \to \infty]{\P} 0, 
        \end{equation*}
        with $\tau_n^{X_n} = \inf \bcb{t \geq 0 \st \nrm{X_n(t)} > n}$. Indeed: 
        \begin{itemize}
            \item if $\tau^X < \tau^Y$, the two quantities coincide for $n$ large enough, by \eqref{eq:cpl1};
            \item if $\tau^Y < \tau^X$, their difference tends to zero due to \eqref{eq:cpl2b};
            \item if $\tau^X = \tau^Y = \infty$, then by \eqref{eq:cpl1} $\tau_n^{X_n}, \tau_n^X \geq \tau_n \to \infty$, so both $\rb{\tau_n^X \wedge T}$ and $\rb{\tau_n^{X_n} \wedge T}$ tend to $T$;
            \item if $\tau^X = \tau^Y < \infty$, we first fix a smaller truncation level and then let it tend to infinity. By \eqref{eq:cpl1}, $\tau_n^{X_n} \geq \tau_n \to \tau^X$. Conversely, for every $m \leq n$ and $t \geq 0$, 
            \begin{align*}
                \hspace{40 pt} & \hspace{-40 pt} \P \brb{\tau^X = \tau^Y \leq t < \tau_n^{X_n}} \\
                & = \P \brb{\tau^X = \tau^Y \leq t < \tau_n^{X_n}, \ \tau_m^{X_n}\leq t} + \P \brb{\tau^X = \tau^Y \leq t < \tau_n^{X_n}, \ \tau_m^{X_n}> t} \\
                & \leq \P \brb{\tau_m^{X_n}\leq t < \tau_n^{X_n}} + \P \brb{\tau_m^X <\tau^X=\tau^Y \leq t < \tau_m^{X_n}}\\
                & \leq \P \brb{\tau_m^{X_n}\leq t < \tau_n^{X_n}} + \P \brb{\tau_n<\tau_m^X< \tau^X=\tau^Y},
            \end{align*}
            where the last inequality follows from the fact that $\tau_m^X<\tau_m^{X_n}$ is only possible if the coupling is broken before $\tau^X_m$, i.e.\ $\tau_n<\tau_m^X$. Taking $n\to\infty$, since $\tau_n$ tends to $\tau^X=\tau^Y$ almost surely,
            \begin{align*}
                \limsup_{n\to\infty} \P \brb{\tau^X = \tau^Y \leq t < \tau_n^{X_n}}&\leq\limsup_{n\to\infty} \P \brb{\tau_m^{X_n}\leq t < \tau_n^{X_n}}\\
                &=\limsup_{n\to\infty} \P \brb{\tau_m^{\hat{X}}\leq t < \tau_n^{\hat{X}}}=\P \brb{\tau_m^{\hat{X}}\leq t < \tau^{\hat{X}}}.
            \end{align*}
            Since the latter tends to 0 as $m$ tends to infinity, by arbitrariness of $m$
            \begin{equation*}
               \limsup_{n\to\infty} \P \brb{\tau^X = \tau^Y \leq t < \tau_n^{X_n}}=0.
            \end{equation*}
            It follows that $\tau_n^{X_n} \wedge T$ converges in probability to $\tau^X \wedge T$ if $\tau^X=\tau^Y<\infty$. Since also $\tau_n^X \to \tau^X$, the desired convergence follows.
        \end{itemize}
        These arguments show that, for every fixed $t \geq 0$, the events $\cb{\tau_n^X \leq t}$ and $\cb{\tau_n^{X_n} \leq t}$ differ with probability tending to zero, which justifies the replacement of the exit times in the third-last equality of \eqref{eq:last}. Finally, summing over $j$ yields \eqref{eq:dist}, and hence the thesis.
    \end{proof}
\end{appendices}

\section*{Acknowledgements}

The project was partially supported by the MUR PRIN grant ``ConStRAINeD'' with number 2022XRWY7W.

\end{document}